\documentclass[11pt,reqno]{amsart}
\usepackage[a4paper,margin=30mm]{geometry}
\usepackage[T1]{fontenc}
\usepackage{lmodern}
\usepackage{amsmath,amssymb,amsthm,mathtools}
\usepackage{microtype}
\usepackage{enumitem}
\usepackage[colorlinks=true]{hyperref}
\usepackage[nameinlink,capitalize]{cleveref}

\numberwithin{equation}{section}

\newtheorem{theorem}{Theorem}[section]
\newtheorem{proposition}[theorem]{Proposition}
\newtheorem{lemma}[theorem]{Lemma}
\newtheorem{corollary}[theorem]{Corollary}
\newtheorem{hypothesis}[theorem]{Hypothesis}
\crefname{hypothesis}{Hypothesis}{Hypotheses}
\Crefname{hypothesis}{Hypothesis}{Hypotheses}
\crefalias{hypothesis}{hypothesis}
\theoremstyle{definition}
\newtheorem{definition}[theorem]{Definition}
\theoremstyle{remark}
\newtheorem{remark}[theorem]{Remark}

\newcommand{\Q}{\mathbb Q}
\newcommand{\Z}{\mathbb Z}
\newcommand{\C}{\mathbb C}
\newcommand{\F}{\mathbb F}
\newcommand{\A}{\mathbb A}

\newcommand{\GQ}{G_{\Q}}
\newcommand{\Gal}{\operatorname{Gal}}
\newcommand{\Frob}{\operatorname{Frob}}
\newcommand{\Ind}{\operatorname{Ind}}
\newcommand{\tr}{\operatorname{tr}}
\newcommand{\Sym}{\operatorname{Sym}}
\newcommand{\im}{\operatorname{im}}

\newcommand{\ST}{\mathrm{ST}}
\newcommand{\CM}{\mathrm{CM}}
\newcommand{\nCM}{\mathrm{nCM}}
\newcommand{\1}{\mathbf 1}
\newcommand{\dd}{\,\mathrm d}
\newcommand{\abs}[1]{\lvert #1\rvert}

\newcommand{\GL}{\mathrm{GL}}
\newcommand{\SU}{\mathrm{SU}}

\title[Upper bounds over polynomial values]{Upper bounds for arithmetic functions over polynomial values via Chebotarev--Sato--Tate distributions}
\author{Jiong Yang}
\address{School of Mathematics and Statistics, Qingdao University, Qingdao, Shandong 266000, People's Republic of China}
\email{yangjiong@qdu.edu.cn}
\date{September 2026}
\subjclass[2020]{11F30, 11N36, 11N37, 11R42}
\keywords{polynomial values, Nair--Tenenbaum theorem, Chebotarev--Sato--Tate, symmetric powers, Rankin--Selberg coefficients, Dedekind zeta-functions, CM forms, congruences of modular forms}

\begin{document}

\begin{abstract}
We establish logarithmic upper bounds for nonnegative multiplicative functions evaluated at values of multivariable polynomials. The polynomial contribution is encoded by the permutation character on the geometric irreducible components of the corresponding hypersurface, while the arithmetic contribution is described by a class function on a joint Chebotarev--Sato--Tate group. This yields a unified framework for symmetric-power coefficients of CM and non-CM modular forms, Dedekind zeta-function coefficients, and mixed automorphic--Galois weights. As a further application, we obtain quantitative divisibility results for Fourier coefficients along polynomial values, including explicit formulas in the full residual-image and Eisenstein congruence cases.
\end{abstract}

\maketitle

\section{Introduction}

The study of arithmetic functions evaluated at polynomial values is a classical theme in analytic number theory.
Throughout this paper, let
\[
Q\in \Z[x_1,\ldots,x_m]
\]
be a primitive square-free polynomial and let $h$ be a nonnegative multiplicative function.
Let
\[
\mathcal B=\prod_{j=1}^m(X_j-Y_j,X_j]\cap\Z^m,
\qquad
V=\prod_{j=1}^mY_j,
\qquad
X^*=\max_jX_j,
\]
where $1\le Y_j\le X_j$. We assume throughout that
\begin{equation*}\tag{B}
\max_j\log X_j\le A\min_j\log Y_j
\end{equation*}
for a fixed $A\ge1$.
We are interested in upper bounds for
\[
\sum_{\substack{\mathbf n\in \mathcal B\\ Q(\mathbf n)\ne 0}} h(\abs{Q(\mathbf n)}).
\]
Such sums arise naturally in the study of divisor functions, representation functions, ideal-counting functions, and Fourier coefficients of automorphic forms.

The subject goes back at least to Erd\H{o}s \cite{Erdos1952}, who studied the divisor function along polynomial values. General multiplicative functions were studied by Shiu \cite{Shiu1980}. Nair \cite{Nair1992} established a general upper-bound theorem for multiplicative functions evaluated at polynomial values in short intervals, and Nair and Tenenbaum \cite{NairTenenbaum1998} substantially extended this theory to more general arithmetic functions and several polynomial arguments.

Analogous questions for binary forms were investigated by de la Bretèche and Browning \cite{BDBrowning2006} and, in greater generality, by de la Bretèche and Tenenbaum \cite{deLaBretecheTenenbaum2012}. More recently, general upper-bound principles for multiplicative functions along suitably equidistributed sequences have been developed by Chan, Koymans, Pagano and Sofos in \cite{CKPSAverages2025}, while de Bretèche and Tenenbaum \cite{BT2026} obtained a flexible multivariable polynomial-value estimate.

Our aim is not to establish another general Nair--Tenenbaum type estimate. Instead, we investigate the arithmetic information contained in the prime-level factor that occurs in such estimates. We show that, for a broad class of arithmetic functions arising from Galois representations and automorphic forms, this factor admits a natural representation-theoretic interpretation in terms of a joint Chebotarev--Sato--Tate distribution. The novelty lies in combining the top-weight permutation representation governing the local densities of $Q$ with the arithmetic prime weight on a single joint group.




Let
\[
X_Q=V(Q)\subset \A^m_{\Q},
\]
and let $M_Q/\Q$ be the minimal finite Galois extension over which all geometric irreducible components of $X_Q$ are defined. Let $G_Q=\Gal(M_Q/\Q)$ be the Galois group acting faithfully on these components. Let $\tau_Q$ denote the corresponding permutation character. If $r(Q)$ denotes the number of irreducible factors of $Q$ over $\Q$, then
\[
\langle \tau_Q,1\rangle_{G_Q}=r(Q).
\]
Moreover, writing $N_Q(p)=|V(Q)(\F_p)|$, we have, at every sufficiently good prime,
\begin{equation}\label{eq:intro-localdensity}
\frac{N_Q(p)}{p^{m-1}}
=\tau_Q(\Frob_p)+O_Q(p^{-1/2}).
\end{equation}
Thus the familiar local density of the polynomial is governed, to first order, by the top-dimensional permutation representation of the hypersurface $V(Q)$.

Now let $M/\Q$ be a finite Galois extension containing $M_Q$, put $G=\Gal(M/\Q)$, and let $S$ be the Sato--Tate group attached to the automorphic or motivic datum under consideration; $S$ is allowed to be trivial. The relevant prime data are carried by a compact Chebotarev--Sato--Tate group
$$
G_S\subseteq G\times S,
$$
with Sato--Tate measure $\dd\mu_{G_S}(g,s)$.
At a good prime $p$, write
$$
j_p=(\Frob_p,s_p)\in G_S.
$$
Suppose that the prime values of $h$ are given by a bounded nonnegative class function $\Phi$:
$$
h(p)=\Phi(j_p).
$$
The fundamental constant occurring in our estimates is
\begin{equation}\label{eq:intro-kappa}
\kappa_Q(\Phi)=
\int_{G_S}\tau_Q(g)\Phi(g,s)
\,\dd\mu_{G_S}(g,s).
\end{equation}
It measures the correlation between the geometric component representation attached to $Q$ and the arithmetic function defining $h$.

Our first result is the following theorem.
\begin{theorem}\label{thm:main1}
Let $Q$ be primitive and square-free, let $h$ be admissible in the sense of \cref{sec:sieve}, and assume that the box $\mathcal B$ satisfies condition {\rm(B)}. Suppose that
$$
h(p)=\Phi(j_p)
$$
at every good prime. If the function
$$
F_Q(g,s)=\tau_Q(g)\Phi(g,s)
$$
satisfies the Mertens-effective Chebotarev--Sato--Tate hypothesis of
\cref{hyp:effectiveCST}, then
\begin{equation}\label{eq:intro-general-bound}
\sum_{\substack{\mathbf n\in\mathcal B\\Q(\mathbf n)\ne0}}
h\bigl(\abs{Q(\mathbf n)}\bigr)
\ll_{Q,h,A}
V(\log X^*)^{\kappa_Q(\Phi)-r(Q)}.
\end{equation}
If only the qualitative Chebotarev--Sato--Tate hypothesis is known, then the exponent on the right-hand side may be replaced by
$\kappa_Q(\Phi)-r(Q)+\varepsilon$
for every $\varepsilon>0$.
\end{theorem}

\subsection{Non-CM modular forms and symmetric powers}
Our first family of applications concerns Fourier coefficients of symmetric-power lifts. Let $\pi$ be the cuspidal automorphic representation of $\GL_2(\A_\Q)$ associated with a primitive non-CM holomorphic newform; for simplicity we assume that its central character is trivial. At every unramified prime write
$$
\lambda_\pi(p)=2\cos\theta_p,
\qquad 0\le\theta_p\le\pi.
$$
The Sato--Tate group is $\SU(2)$ and
$$
\dd\mu_{\ST}(\theta)=\frac2\pi\sin^2\theta\,\dd\theta.
$$
For $d\ge1$,
$$
\lambda_{\Sym^d\pi}(p)=U_d(\cos\theta_p)=\frac{\sin((d+1)\theta_p)}{\sin\theta_p}.
$$

Define
\begin{equation}\label{eq:intro-kappad}
\kappa_d^{\nCM}=\frac2\pi\int_0^\pi\abs{U_d(\cos\theta)}\sin^2\theta\,\dd\theta.
\end{equation}
We have
$$
\kappa_Q(H_d)=r(Q)\kappa_d^{\nCM},
\qquad
H_d(\theta)=\abs{U_d(\cos\theta)}.
$$

Applying \cref{thm:main1}, we obtain the following result.

\begin{theorem}\label{thm:main2}
    Assume that the effective Chebotarev--Sato--Tate hypothesis holds for $(\pi,Q)$. Then
\begin{equation}
\sum_{\substack{\mathbf n\in\mathcal B\\Q(\mathbf n)\ne0}}
\abs{\lambda_{\Sym^d\pi}(\abs{Q(\mathbf n)})}
\ll V(\log X^*)^{-r(Q)(1-\kappa_d^{\nCM})}.
\end{equation}
If only the qualitative version is known, the exponent on the right-hand side may be increased by an arbitrary $\varepsilon>0$.
\end{theorem}

The Chebotarev--Sato--Tate phenomenon was studied in \cite{MurtyMurty2009,Wong2019}. Their results imply that, for non-CM holomorphic cusp forms, the Chebotarev--Sato--Tate theorem holds in the following cases.
\begin{itemize}
    \item $M_Q$ is solvable over $\Q$.
    \item $G_Q=\Gal(M_Q/\Q)\simeq G_1\times G_2$ with $M_Q^{G_1}$ totally real and $G_1$ abelian.
    \item The strong Artin conjecture holds for $M_Q$, for example for extensions of degree at most $4$ and in many degree-$5$ cases.
\end{itemize}
In particular, if $Q$ is geometrically irreducible, then $M_Q=\Q$, and \cref{thm:main2} follows unconditionally.

For the effective version, Thorner's work in \cite{Thorner2021,Thorner2025} implies that the required effective Chebotarev--Sato--Tate estimates hold when $M/\Q$ is abelian or totally real.

When $m=1$ and $Q$ is irreducible, Chiriac and Yang \cite{ChiriacYang2022} obtained an upper bound under the strong Artin conjecture for $M_Q$. Their result was later refined by Luo and Lao \cite{LuoLao2024} using the Sato--Tate theorem; in the present notation, their bound is
\[
\sum_{X-Y<n\le X}
\abs{a_{\Sym^d\pi}(\abs{Q(n)})}
\ll Y(\log X)^{\kappa_d^{\nCM}-1}.
\]
The logarithmic exponent is the same as ours, but we only require the Chebotarev--Sato--Tate conjecture rather than the strong Artin conjecture.

Woo \cite{Woo2026} also studied the corresponding absolute-value sum for a one-variable polynomial $Q$. Assuming that the base change of $\pi$ to $M_Q$ is cuspidal, she proved a bound of the form
\[
\sum_{n\le X}\abs{a_{\pi}(\abs{Q(n)})}
\ll X(\log X)^{-0.066}.
\]
Her assumption is slightly weaker than ours, since the Chebotarev--Sato--Tate conjecture follows if the cuspidal base change exists for all $\Sym^d\pi$, whereas our resulting estimate is stronger.

When $m>1$, Chiriac and Yang \cite{ChiriacYang2022} also treated cubic polynomials in two variables defining elliptic curves. Our theorem extends this result to arbitrary primitive square-free polynomials.

We also consider the Rankin--Selberg case. For any $a,b\geq1$, define
\[\kappa_{a,b}^{\Delta}
=\frac2\pi\int_0^\pi
\abs{\sin((a+1)\theta)\sin((b+1)\theta)}\,\dd\theta.\]

These constants make the logarithmic saving explicit. Since $U_d$ is the character of a nontrivial irreducible representation of $\mathrm{SU}(2)$, character orthogonality and the strict Cauchy--Schwarz inequality give
\[
0<\kappa_d^{\nCM}<1 \qquad (d\ge1).
\]
Similarly,
\[
\kappa_{a,b}^{\Delta}\le1,
\]
with equality if and only if $a=b$. Thus the symmetric-power estimate above gives a genuine logarithmic saving for every $d\ge1$, while the Rankin--Selberg-type estimate below gives a genuine logarithmic saving when $a\ne b$.

\begin{theorem}
  Assume that the effective Chebotarev--Sato--Tate hypothesis holds for $(\pi,Q)$. Then
  \[\sum_{\substack{\mathbf n\in\mathcal B\\Q(\mathbf n)\ne0}}
\abs{\lambda_{\Sym^a\pi}(\abs{Q(\mathbf n)})
\lambda_{\Sym^b\pi}(\abs{Q(\mathbf n)})}
\ll V(\log X^*)^{-r(Q)(1-\kappa_{a,b}^{\Delta})}.\]
\end{theorem}

Let $\pi_1$ and $\pi_2$ be two cuspidal automorphic representations of $\mathrm{GL}_2(\A_\Q)$ corresponding to non-CM holomorphic modular forms. The relevant Sato--Tate group is $S=\mathrm{SU}(2)\times\mathrm{SU}(2)$, and Hypothesis~\ref{hyp:effectiveCST} becomes the effective joint Chebotarev--Sato--Tate hypothesis. We have the following theorem.

\begin{theorem}
Assume that $\pi_1$ and $\pi_2$ are not twist-equivalent and  $(Q,\pi_1,\pi_2)$ satisfies Hypothesis~\ref{hyp:effectiveCST}. Then
\begin{equation}\label{eq:twoformsmain}
\sum_{\substack{\mathbf n\in\mathcal B\\Q(\mathbf n)\ne0}}
\abs{\lambda_{\Sym^a\pi_1}(\abs{Q(\mathbf n)})
\lambda_{\Sym^b\pi_2}(\abs{Q(\mathbf n)})}
\ll V(\log X^*)^{-r(Q)(1-\kappa_a^{\nCM}\kappa_b^{\nCM})}.
\end{equation}
\end{theorem}

By \cite{Thorner2021,Thorner2025}, the effective joint Chebotarev--Sato--Tate estimates required here are unconditional when $M/\Q$ is abelian or totally real.

\subsection{CM forms and arithmetic entanglement}

The CM case reveals a phenomenon that is invisible in the non-CM setting. Let $\pi$ be a primitive CM newform with CM by an imaginary quadratic field $E$. Its Sato--Tate group is the disconnected normalizer
$$
S_{\CM}=N_{\SU(2)}(T)=T\sqcup wT,
$$
where $T\simeq\mathrm U(1)$. If
$$
\chi_E:G_\Q\longrightarrow\{\pm1\}
$$
is the quadratic character of $E/\Q$ and
$$
c:S_{\CM}\longrightarrow\{\pm1\}
$$
is the component character, then the Frobenius components satisfy
$$
c(s_p)=\chi_E(p).
$$
Thus the joint group is not a direct product but the fiber product
$$
G_S^{\CM}=
\{(\sigma,s)\in G\times S_{\CM}:
\chi_E(\sigma)=c(s)\}.
$$

This interaction produces an additional correlation term. Put
$$
\eta_Q(E)=\langle\tau_Q,\chi_E\rangle_G.
$$
For
$$
A_d=\frac1\pi\int_0^\pi
\abs{U_d(\cos\theta)}\,\dd\theta,
\qquad
B_d=\abs{U_d(0)},
$$
we obtain
\begin{equation}\label{eq:intro-CMconstant}
\kappa_Q(H_d)=\frac{r(Q)+\eta_Q(E)}2A_d+\frac{r(Q)-\eta_Q(E)}2B_d.
\end{equation}
The required reciprocal-prime estimate is unconditional in the CM case and follows from classical Hecke $L$-function theory. Thus we obtain the following theorem.
\begin{theorem}\label{thm:main4}
    Let $\pi$ be as above. Then
\begin{equation}
\sum_{\substack{\mathbf n\in\mathcal B\\Q(\mathbf n)\ne0}}
\abs{\lambda_{\Sym^d\pi}(\abs{Q(\mathbf n)})}
\ll V(\log X^*)^{\kappa_Q(H_d)-r(Q)},
\end{equation}
where $\kappa_Q(H_d)$ is given by \eqref{eq:CMHdconstant}.
\end{theorem}

When $Q$ is irreducible, $\eta_Q(E)\in\{0,1\}$. It equals $1$ precisely when the CM character becomes trivial on the stabilizer of a geometric component, equivalently when the CM field occurs in the corresponding component field. Thus \eqref{eq:intro-CMconstant} gives a particularly transparent example of the arithmetic entanglement detected by our formalism.

Let $\pi$ be an automorphic representation of $\mathrm{GL}_2(\A_\Q)$ attached to a holomorphic modular form, and let $q(x) \in \Z[x]$ be a monic quadratic polynomial. In \cite{blomersums2008}, Blomer studied $\sum_{n\leq x}a_\pi(q(n))$ and proved that
\[
\sum_{n \le x} a_\pi(q(n)) = c x + O\left(x^{6/7+\varepsilon}\right).
\]
He gave an explicit expression for $c$ via trace formulas of Hecke operators. He showed that $c = 0$ if the weight is even or the discriminant of $q(x)$ is positive, but a simple vanishing criterion was missing. The vanishing of $c$ was discussed in \cite{kuan2023sums}, where it was conjectured that $c = 0$ unless $\pi$ is dihedral. Combining our \cref{thm:main2} and \cref{thm:main4}, we confirm this conjecture as the following corollary.
\begin{corollary}
    Let $\pi$ be an automorphic representation of $\mathrm{GL}_2(\A_\Q)$ attached to a holomorphic modular form, and let $q(x) \in \Z[x]$ be a monic quadratic polynomial. Then
\[
\sum_{n \le x} a_\pi(q(n)) = c x + O\left(x^{6/7+\varepsilon}\right).
\]
The constant $c$ vanishes if either of the following holds:
\begin{itemize}
    \item $\pi$ is not a CM modular form, or
    \item $\pi$ is a CM modular form whose associated field $K_\pi$ is distinct from the quadratic field defined by $q$.
\end{itemize}

\end{corollary}

\subsection{Dedekind zeta-functions and mixed weights}

Another natural example is the ideal-counting function. Let $K/\Q$ be a number field and write
$$
\zeta_K(s)=\sum_{n\ge1}\frac{a_K(n)}{n^s}.
$$
At every unramified prime, $a_K(p)$ is itself a permutation character. More precisely, after enlarging $M$ to contain the Galois closure of $K$, let
$$
H_K=\Gal(M/K),
\qquad
\tau_K=\Ind_{H_K}^G1.
$$
Then
$$
a_K(p)=\tau_K(\Frob_p).$$
The relevant logarithmic constant therefore has the purely representation-theoretic form
\begin{equation}\label{eq:intro-kappaK}
\kappa_Q(K)=\langle\tau_Q,\tau_K\rangle_G.
\end{equation}

Our theorem gives the following unconditional estimate.
\begin{theorem}
We have
\begin{equation}\label{eq:intro-Dedekind}
\sum_{\substack{\mathbf n\in\mathcal B\\Q(\mathbf n)\ne0}}
a_K(\abs{Q(\mathbf n)})
\ll_{Q,K,A}V(\log X^*)^{\kappa_Q(K)-r(Q)}.
\end{equation}
\end{theorem}

A classical special case is the representation function $r_2(n)$. Indeed,
$$
\sum_{n\ge1}\frac{r_2(n)}{n^s}=4\zeta_{\Q(i)}(s),
$$
so that
$$
r_2(n)=4a_{\Q(i)}(n).
$$

Related lattice-point problems involving quadratic forms were studied by Fomenko \cite{Fomenko1997}, while Kim \cite{Kim2007} treated more general polynomial congruence problems under strong Artin-type hypotheses. Formula \eqref{eq:intro-Dedekind} applies to the ideal-counting function of an arbitrary number field and requires no Artin conjecture. If $Q$ is irreducible and $M_Q$ is linearly disjoint from the Galois closure of $K$, then
$$
\kappa_Q(K)=1.
$$

The class-function formulation also allows finite Galois and automorphic weights to be multiplied. For a non-CM newform $\pi$, we obtain, under the corresponding effective hybrid distribution,
$$
\sum_{\substack{\mathbf n\in\mathcal B\\Q(\mathbf n)\ne0}}
a_K(\abs{Q(\mathbf n)})
\abs{\lambda_{\Sym^d\pi}(\abs{Q(\mathbf n)})}
\ll
V(\log X^*)^{\kappa_Q(K)\kappa_d^{\nCM}-r(Q)}.
$$
This mixed example illustrates the advantage of formulating the prime input on a single joint group rather than treating Galois and automorphic weights separately.

\subsection{Congruences of modular forms}\label{subsec:intro-congruences}

Our final application concerns divisibility of Fourier coefficients along polynomial values. Classical work of Serre \cite{Serre1976} shows that divisibility phenomena for coefficients of modular forms are extremely abundant; for a fixed modulus, one often obtains divisibility for almost all Fourier coefficients. We ask for a polynomial-value analogue and, more importantly, for a quantitative description of the exceptional set in terms of the interaction between the polynomial and the residual Galois representation.

Let
$$
f(z)=\sum_{n\ge1}a_f(n)q^n
$$
be a normalized cuspidal Hecke eigenform of weight $k\ge2$, level $N$, and nebentypus $\chi$. Let $\lambda$ be a prime of its coefficient field above the rational prime $\ell$, and let
$$
\bar\rho_{f,\lambda}:
G_\Q\longrightarrow\GL_2(k_\lambda)
$$
be the associated residual representation. At every prime $p\nmid N\ell$,
$$
\tr\bar\rho_{f,\lambda}(\Frob_p)
\equiv a_f(p)\pmod\lambda.
$$
Let $L_{f,\lambda}$ be the field cut out by $\bar\rho_{f,\lambda}$, put
$$
G=\Gal(M_QL_{f,\lambda}/\Q),
$$
and define
$$
z_{f,\lambda}(\sigma)=\mathbf 1_{\{\tr\bar\rho_{f,\lambda}(\sigma)=0\}}.
$$
The constant controlling divisibility along $Q$ is
\begin{equation}\label{eq:intro-delta}
\delta_{Q,f,\lambda}=
\langle\tau_Q,z_{f,\lambda}\rangle_G.
\end{equation}
We prove the following polynomial analogue of Serre-type divisibility.

\begin{theorem}\label{thm:intro-divisibility}
For every primitive square-free polynomial $Q$,
\begin{equation}\label{eq:intro-divisibility-bound}
\#\{
\mathbf n\in\mathcal B:
Q(\mathbf n)\ne0,
\lambda\nmid a_f(\abs{Q(\mathbf n)})
\}
\ll_{Q,f,\lambda,A}
\frac{V}
{(\log X^*)^{\delta_{Q,f,\lambda}}}.
\end{equation}
In particular, if $\delta_{Q,f,\lambda}>0$, then
$$
\lambda\mid a_f(\abs{Q(\mathbf n)})
$$
for all but a logarithmically sparse set of polynomial parameters.
\end{theorem}

\begin{remark}
\cref{thm:intro-divisibility} is unconditional: the relevant prime distribution is purely Galois and therefore follows from the effective Chebotarev density theorem.
\end{remark}

The exponent in \eqref{eq:intro-divisibility-bound} has a concrete Galois-theoretic interpretation. If
$$
M_Q\cap L_{f,\lambda}=\Q,
$$
then the two finite Galois data are independent and
\begin{equation}\label{eq:intro-deltaLD}
\delta_{Q,f,\lambda}=
r(Q)\vartheta_{f,\lambda},
\qquad
\vartheta_{f,\lambda}=
\frac{
\#\{g\in\im\bar\rho_{f,\lambda}:\tr g=0\}
}{
\abs{\im\bar\rho_{f,\lambda}}
}.
\end{equation}
For instance, if
$$\im\bar\rho_{f,\lambda}=\GL_2(\F_q),\qquad q=|k_\lambda|,
$$
then
$$
\#\{g\in\GL_2(\F_q):\tr g=0\}=q^2(q-1),
$$
and hence
$$
\vartheta_{f,\lambda}
=
\frac{q}{q^2-1}.
$$
Thus
\begin{equation}\label{eq:intro-full-image}
\#\{
\mathbf n\in\mathcal B:
Q(\mathbf n)\ne0,
\lambda\nmid a_f(\abs{Q(\mathbf n)})
\}
\ll
\frac{V}
{(\log X^*)^{r(Q)q/(q^2-1)}}.
\end{equation}
The same statement applies, in particular, to elliptic curves whose mod-$\ell$ representation is surjective.

The reducible residual case leads to an especially explicit formula. Suppose that $f$ is congruent modulo $\lambda$ to an Eisenstein eigensystem:
$$
a_f(p)
\equiv
\psi_1(p)+\psi_2(p)p^{k-1}
\pmod\lambda
$$
for all but finitely many primes $p$. Then
$$
\bar\rho_{f,\lambda}^{\mathrm{ss}}
\simeq
\chi_1\oplus\chi_2,
$$
and the trace-zero condition is controlled by the ratio character
$$
\xi=\chi_2\chi_1^{-1}.
$$
Indeed,
$$
\tr\bar\rho_{f,\lambda}(\sigma)=0
\quad\Longleftrightarrow\quad
\xi(\sigma)=-1.
$$
Let
$$
m_\xi=\abs{\im\xi}.
$$

If $M_Q$ is linearly disjoint from the field $L_\xi$ cut out by $\xi$, then
\begin{equation}\label{eq:intro-eisenstein-delta}
\delta_{Q,f,\lambda}
=\begin{cases}
r(Q)/m_\xi,&m_\xi\ \text{even},\\
0,&m_\xi\ \text{odd}.
\end{cases}
\end{equation}
More significantly, the constant remains completely explicit without any linear-disjointness assumption. If
$$
Q=R_1\cdots R_{r(Q)}
$$
is the factorization into irreducible factors over $\Q$, if $H_j$ denotes the stabilizer of a geometric component of $V(R_j)$ in
$$
G_\xi=\Gal(M_QL_\xi/\Q),
$$
and
$$
m_j=\abs{\im(\xi|_{H_j})},
$$
then
\begin{equation}\label{eq:intro-entangled}
\delta_{Q,f,\lambda}=
\sum_{\substack{1\le j\le r(Q)\\ m_j\ {\rm even}}}
\frac1{m_j}.
\end{equation}
Thus the exponent itself records the precise entanglement between the component fields of the polynomial and the Eisenstein character.

Two familiar congruences give concrete illustrations. Ramanujan's congruence
$$
\tau(p)\equiv1+p^{11}\pmod{691}
$$
corresponds to
$$
\xi=\bar\chi_{691}^{\,11}.
$$
Since $\gcd(11,690)=1$, its image has order $690$. Hence, if
$$
M_Q\cap\Q(\zeta_{691})=\Q,
$$
then
$$
\delta_{Q,\Delta,691}=\frac{r(Q)}{690}
$$
and
\begin{equation}\label{eq:intro-Ramanujan}
\#\{
\mathbf n\in\mathcal B:
Q(\mathbf n)\ne0,
691\nmid\tau(\abs{Q(\mathbf n)})
\}
\ll
\frac{V}{(\log X^*)^{r(Q)/690}}.
\end{equation}
Thus Ramanujan's divisibility congruence holds along polynomial values outside a quantitatively controlled logarithmically sparse exceptional set.

As a second example, consider
$$
f_{11}(z)=\eta(z)^2\eta(11z)^2
=\sum_{n\ge1}a_{11}(n)q^n.
$$
Writing
$$
\prod_{m\ge1}(1-q^m)^2(1-q^{11m})^2
=
\sum_{n\ge0}v(n)q^n,
$$
we have
$$
a_{11}(n)=v(n-1).
$$

The integers $v(n)$ may be interpreted as the difference
$$
v(n)=v_e(n)-v_o(n)
$$
between the numbers of certain four-colored partitions with an even and an odd number of parts. Since
$$
a_{11}(p)\equiv p+1\pmod5,
$$
the associated ratio character is $\bar\chi_5$. Therefore, if
$$
M_Q\cap\Q(\zeta_5)=\Q,
$$
then
$$
\delta_{Q,f_{11},5}=\frac{r(Q)}4
$$
and
\begin{equation}\label{eq:intro-colored}
\#\{
\mathbf n\in\mathcal B:
Q(\mathbf n)\ne0,
5\nmid v(\abs{Q(\mathbf n)}-1)
\}
\ll_Q
\frac{V}{(\log X^*)^{r(Q)/4}}.
\end{equation}
Equivalently,
$$
v_e(\abs{Q(\mathbf n)}-1)
\equiv
v_o(\abs{Q(\mathbf n)}-1)
\pmod5
$$
for a density-one set of polynomial parameters.


\subsection{Organization of the paper}

The paper is organized as follows. In \cref{sec:local} we study the local densities of $Q$ and identify their leading term with the permutation character $\tau_Q$ on the geometric irreducible components of $V(Q)$. In \cref{sec:sieve} we combine these local estimates with the multivariable theorem of de la Bretèche and Tenenbaum to prove the general polynomial-value upper bound. In \cref{sec:CST} we introduce the joint Chebotarev--Sato--Tate framework, formulate the qualitative and Mertens-effective distributions, and discuss cases in which they are known. In \cref{sec:applications} we apply the general theorem to non-CM and CM symmetric-power coefficients, Dedekind zeta-function coefficients, and mixed automorphic--Galois weights. The final section is devoted to congruences of modular forms, including the full residual-image case, Eisenstein congruences, Ramanujan's congruence modulo $691$, and the colored-partition example described above.

\section{Polynomial local densities}\label{sec:local}

Let $m\ge1$ and let $Q\in\Z[x_1,\ldots,x_m]$ be primitive and square-free over $\Q$. For any integer $q\ge1$, define
\[
N_Q(q)=\#\{\mathbf x\bmod q:Q(\mathbf x)\equiv0\pmod q\},
\qquad
\delta_Q(q)=\frac{N_Q(q)}{q^m},
\]
and, for primes $p$,
\[
\omega_Q(p)=p\delta_Q(p)=\frac{N_Q(p)}{p^{m-1}}.
\]
Write
\[
Q=R_1\cdots R_r
\]
for the factorization of $Q$ into distinct primitive irreducible factors over $\Q$, and put $r(Q)=r$.

\subsection{The top-weight component representation}

Let
\[
X_Q:=V(Q)\subset \A^m_{\Q}.
\]
Since $Q$ is square-free, $X_Q$ is reduced and has pure dimension $d=m-1$. Over $\overline{\Q}$, write
\[
Q=cQ_1\cdots Q_t,
\]
where the $Q_i$ are the distinct geometric irreducible factors. Define the \emph{component field} $M_Q$ to be the fixed field of the kernel of the natural action
\[
G_{\Q}\longrightarrow\operatorname{Perm}\bigl(\operatorname{Irr}((X_Q)_{\overline{\Q}})\bigr).
\]
Thus $M_Q/\Q$ is the minimal finite Galois extension over which all geometric components are defined. Put
\[
G_Q=\Gal(M_Q/\Q),
\]
and let $\tau_Q$ be the resulting faithful permutation character of $G_Q$ on the geometric irreducible components of $(X_Q)_{\overline{\Q}}$.

When $m=1$ and $Q$ is irreducible, let $\alpha$ be a root of $Q$ with $F=\Q(\alpha)$. Let $M_Q$ be the splitting field of $Q$, and write $H=\Gal(M_Q/F)$. Then the corresponding permutation character is
\[
\tau_Q=\Ind_H^{G_Q}\1.
\]

\begin{lemma}\label{lem:topweight}
For every sufficiently large prime $p$, we have
\begin{equation}\label{eq:NQp}
N_Q(p)=p^{m-1}\tau_Q(\Frob_p)+O_Q(p^{m-3/2}).
\end{equation}
Equivalently,
\begin{equation}\label{eq:omegaFrob}
\omega_Q(p)=\tau_Q(\Frob_p)+O_Q(p^{-1/2}).
\end{equation}
If $m=1$, then at every sufficiently good prime the identity in \eqref{eq:omegaFrob} is exact. Moreover,
\[
\langle\tau_Q,1\rangle_{G_Q}=r(Q),
\]
and consequently
\begin{equation}\label{eq:omegaMertens}
\sum_{p\le x}\frac{\omega_Q(p)}p=r(Q)\log\log x+O_Q(1).
\end{equation}
\end{lemma}

\begin{proof}
Assume first that $m\ge2$ and put $d=m-1$. Serre's top-weight description of $N_X(p)$ in \cite[Chapter 7]{SerreNXp} shows that, for a reduced $d$-dimensional variety $X/\Q$, the weight-$2d$ part of the point-counting representation is the Tate twist of the permutation representation on the $d$-dimensional geometric irreducible components. In particular, \cite[Proposition 7.10 and Corollary 7.11]{SerreNXp} give, outside a finite set of primes,
\begin{equation}\label{eq:SerreNX}
N_X(p)=\varepsilon_X(\Frob_p)p^d+O_X(p^{d-1/2}),
\end{equation}
where $\varepsilon_X$ is the permutation character on the top-dimensional geometric components. Applying \eqref{eq:SerreNX} with $X=X_Q$ gives \eqref{eq:NQp} and \eqref{eq:omegaFrob}.

If $m=1$, the geometric components are the distinct roots of $Q$. At every sufficiently good prime, $N_Q(p)$ is exactly the number of geometric roots fixed by Frobenius, hence $N_Q(p)=\tau_Q(\Frob_p)$.

By Burnside's lemma,
\[
\langle\tau_Q,1\rangle_{G_Q}
=\#\bigl(G_Q\backslash\{Q_1,\ldots,Q_t\}\bigr).
\]
The Galois orbits of geometric components are in bijection with the irreducible factors of $Q$ over $\Q$, and therefore
\[
\langle\tau_Q,1\rangle_{G_Q}=r(Q).
\]
Finally, the Chebotarev density theorem, in its effective form for the fixed extension $M_Q/\Q$, and partial summation give
\[
\sum_{p\le x}\frac{\tau_Q(\Frob_p)}p=r(Q)\log\log x+O_Q(1).
\]
The difference between this sum and the corresponding sum with $\omega_Q(p)$ is absolutely convergent by \eqref{eq:omegaFrob}. This proves \eqref{eq:omegaMertens}.
\end{proof}

\begin{remark}
If $Q$ is irreducible over $\Q$, then the Galois action on its geometric components is transitive, so $r(Q)=1$ even when $Q$ is geometrically reducible.
\end{remark}

\subsection{Prime-power densities}

The preceding subsection describes the first-order density modulo $p$. We next record the estimates for higher powers of $p$ that will be used in the sieve argument.

\begin{lemma}\label{lem:primepowers}
Let $D=\deg Q$. For every prime $p$ and every $\nu\ge2$,
\begin{equation}\label{eq:primepowers}
\delta_Q(p^2)\ll_Q p^{-2},
\qquad
\delta_Q(p^\nu)\ll_Q (\nu+1)^{m-1}p^{-\nu/D}.
\end{equation}
Consequently, for every fixed $B\ge0$,
\begin{equation}\label{eq:primepowertail}
\sum_{\nu\ge2}(\nu+1)^B\delta_Q(p^\nu)\ll_{Q,B}p^{-2}.
\end{equation}
\end{lemma}

\begin{proof}
Outside a finite set of primes, the reduction of $Q$ remains square-free. The argument of \cite[Lemma 2.6]{CKPS2025} gives
\[
N_Q(p^2)\ll_Q p^{2m-2}.
\]
Indeed, for $m\ge2$ the nonsingular locus modulo $p$ has $O_Q(p^{m-1})$ points, each with exactly $p^{m-1}$ lifts modulo $p^2$, while the singular locus has $O_Q(p^{m-2})$ points (for instance by the Lang--Weil estimate applied to the singular subvariety) \cite{LangWeil1954} and even the trivial bound of $p^m$ lifts over each such point contributes only $O_Q(p^{2m-2})$. For $m=1$, ordinary Hensel lifting gives the same normalized estimate. Thus $\delta_Q(p^2)\ll_Qp^{-2}$.

For general prime powers, \cite[Lemma 2.8]{CKPS2025}, or the explicit form \cite[Lemma 4.2]{BT2026}, gives
\[
\delta_Q(p^\nu)\ll_Q(\nu+1)^{m-1}p^{-\nu/D}.
\]
Choose a fixed $\nu_0>3D$. For $2\le\nu\le\nu_0$ we have
\[
\delta_Q(p^\nu)\le\delta_Q(p^2)\ll_Qp^{-2},
\]
while for $\nu>\nu_0$ the second estimate in \eqref{eq:primepowers} gives a uniformly convergent geometric tail. This proves \eqref{eq:primepowertail}.
\end{proof}

\section{A unified polynomial-value upper bound}\label{sec:sieve}

Let
\[
\mathcal B=\prod_{j=1}^m(X_j-Y_j,X_j]\cap\Z^m,
\qquad
V=\prod_{j=1}^mY_j,
\qquad
X^*=\max_jX_j,
\]
where $1\le Y_j\le X_j$ as before.

\begin{definition}
A nonnegative multiplicative function $h$ is called \emph{admissible} if $h(p)$ is uniformly bounded and there exists a fixed $B_0\ge0$ such that
\begin{equation}\label{eq:admissible}
h(p^\nu)\ll_h(\nu+1)^{B_0}
\qquad(p\text{ prime},\ \nu\ge0).
\end{equation}
\end{definition}

\begin{lemma}\label{lem:admissibleM1}
Every admissible function belongs, after choosing $\varepsilon>0$ sufficiently small, to a class $\mathcal M_1(A_0,B_\varepsilon,\varepsilon)$ of de la Bret\`eche--Tenenbaum \cite[Section 2]{BT2026}.
\end{lemma}

\begin{proof}
Let $h$ be admissible. Multiplicativity and \eqref{eq:admissible} give, for $(a,b)=1$,
\[
h(ab)\le h(b)\prod_{p^\nu\parallel a} C_h(\nu+1)^{B_0}.
\]
Since $\nu+1\le 2^\nu$, the product is at most $A_0^{\Omega(a)}h(b)$ for a suitable constant $A_0$. On the other hand, the standard divisor bound (together with $C_h^{\omega(a)}\ll_{h,\varepsilon}a^{\varepsilon/2}$) gives, for every fixed $\varepsilon>0$,
\[
C_h^{\omega(a)}\prod_{p^\nu\parallel a}(\nu+1)^{B_0}\ll_{h,\varepsilon}a^\varepsilon.
\]
Hence
\[
h(ab)\le \min\{A_0^{\Omega(a)},B_\varepsilon a^\varepsilon\}h(b),
\]
which is precisely the growth condition required in the definition of $\mathcal M_1$.
\end{proof}

Our first result is the following theorem.

\begin{theorem}\label{thm:unified}
Let $m\ge1$, let $Q\in\Z[x_1,\ldots,x_m]$ be primitive and square-free over $\Q$, and let $h$ be admissible. Under (B),
\begin{equation}\label{eq:unified}
\sum_{\substack{\mathbf n\in\mathcal B\\Q(\mathbf n)\ne0}}
h(\abs{Q(\mathbf n)})
\ll_{Q,h,A}
V\exp\left\{\sum_{p\le X^*}\frac{\omega_Q(p)(h(p)-1)}p\right\}.
\end{equation}
All fixed bad primes are absorbed into the implied constant.
\end{theorem}

\begin{proof}
Write
\[
Q=R_1\cdots R_r
\]
as the product of its distinct primitive irreducible factors over $\Q$. Apply \cite[Theorem 3.1]{BT2026} with one polynomial $Q$ and one arithmetic function $F=h$. Since $Q$ is square-free, all exponents of the irreducible factors in the notation of that theorem are one, and hence
\[
\widehat F(s_1,\ldots,s_r)=h(s_1\cdots s_r).
\]
By \cref{lem:admissibleM1}, the arithmetic function lies in the required class $\mathcal M_1$. Condition (B) implies the size hypotheses of \cite[Theorem 3.1]{BT2026}: one may take $\alpha=1/A$ in $X_j^\alpha\le Y_j$ and any fixed $0<\beta<1/A$ in the lower bound for $\min_j X_j$ relative to $\max_j X_j$ (the finitely many bounded values of $X^*$ being absorbed into the implied constant). We then choose the parameter $\varepsilon$ in $\mathcal M_1$ sufficiently small to satisfy the restriction in \cite[(3.1)]{BT2026}. We obtain
\begin{equation}\label{eq:BTapplication}
\sum_{\substack{\mathbf n\in\mathcal B\\Q(\mathbf n)\ne0}}
h(\abs{Q(\mathbf n)})
\ll V E_Q(v)\prod_{g<p\le X_0}(1-\delta_Q(p)),
\end{equation}
where
\[
X_0=\min_jX_j,
\qquad
v=X_1+\cdots+X_m,
\qquad
g=\deg Q,
\]
and
\[
E_Q(v)=
\sum_{\substack{\mathbf s=(s_1,\ldots,s_r)\in\mathbf N^r\\s_1\cdots s_r\le v}}
 h(s_1\cdots s_r)\frac{\rho_Q^\#(\mathbf s)}{K(\mathbf s)^m}.
\]
Here
\begin{equation}\label{eq:Kdef}
K(\mathbf s)=[s_1\kappa(s_1),\ldots,s_r\kappa(s_r)],
\end{equation}
where $[\,\cdot\,]$ denotes least common multiple and $\kappa(s)$ is the square-free kernel of $s$, while
\begin{equation}\label{eq:rhodef}
\rho_Q^\#(\mathbf s)
:=
\sum_{\substack{\boldsymbol\xi\in[1,K(\mathbf s)]^m\\
 s_j\parallel R_j(\boldsymbol\xi),\
 (R_j(\boldsymbol\xi)/s_j,\prod_i s_i)=1\ (1\le j\le r)}}1.
\end{equation}

We make the local structure of $E_Q$ explicit. Fix a prime $p$ outside the finite exceptional set coming from the chosen integral factorization and the bad reductions of the $R_i$. Let
\[
\mathbf e=(e_1,\ldots,e_r)\in\Z_{\ge0}^r,
\qquad
\abs{\mathbf e}=e_1+\cdots+e_r,
\qquad
M(\mathbf e)=\max_i e_i.
\]
For $\mathbf e\ne0$, \eqref{eq:Kdef} gives
\[
K(p^{e_1},\ldots,p^{e_r})=p^{M(\mathbf e)+1}.
\]
Moreover, \eqref{eq:rhodef} implies that
\[
D_p(\mathbf e):=
\frac{\rho_Q^\#(p^{e_1},\ldots,p^{e_r})}{p^{m(M(\mathbf e)+1)}}
\]
is precisely the $p$-adic density of the exact valuation conditions
\[
v_p(R_i(\mathbf x))=e_i\qquad(1\le i\le r).
\]
The coprimality condition in \eqref{eq:rhodef} is important here: if $e_i=0$ while $p$ occurs in another coordinate of $\mathbf e$, it also forces $p\nmid R_i(\mathbf x)$.

Since $Q=R_1\cdots R_r$ is square-free,
\[
v_p(Q(\mathbf x))=\sum_{i=1}^rv_p(R_i(\mathbf x)).
\]
Hence the exact valuation vectors with $\abs{\mathbf e}=\nu$ partition the set $v_p(Q)=\nu$, and therefore
\begin{equation}\label{eq:valuationpartition}
\sum_{\abs{\mathbf e}=\nu}D_p(\mathbf e)
=\delta_Q(p^\nu)-\delta_Q(p^{\nu+1}),
\qquad \nu\ge1.
\end{equation}

All summands in $E_Q$ are nonnegative. Thus we may remove the product-size condition $s_1\cdots s_r\le v$ while retaining the prime-support condition $P^+(s_1\cdots s_r)\le v$, where $P^+(n)$ denotes the largest prime factor of $n$. Multiplicativity and the Chinese remainder theorem then give
\[
E_Q(v)\le\prod_{p\le v}L_p,
\]
where, by \eqref{eq:valuationpartition},
\begin{equation}\label{eq:Lp}
L_p=1+\sum_{\nu\ge1}h(p^\nu)
\bigl(\delta_Q(p^\nu)-\delta_Q(p^{\nu+1})\bigr).
\end{equation}
Separating $\nu=1$ and using \cref{lem:primepowers},
\[
h(p)\bigl(\delta_Q(p)-\delta_Q(p^2)\bigr)
=h(p)\delta_Q(p)+O_{Q,h}(p^{-2}),
\]
while positivity, admissibility, and \eqref{eq:primepowertail} give
\[
\sum_{\nu\ge2}h(p^\nu)
\bigl(\delta_Q(p^\nu)-\delta_Q(p^{\nu+1})\bigr)
\ll_{Q,h}p^{-2}.
\]
Consequently,
\[
L_p=1+h(p)\delta_Q(p)+O_{Q,h}(p^{-2}),
\]
and, for primes in the range of the sieve product,
\begin{equation}\label{eq:sievelocal}
(1-\delta_Q(p))L_p
=1+(h(p)-1)\delta_Q(p)+O_{Q,h}(p^{-2}).
\end{equation}

It remains to compare the cutoffs. From \eqref{eq:BTapplication} and \eqref{eq:sievelocal},
\begin{align}\label{eq:cutoffcompare}
E_Q(v)\prod_{g<p\le X_0}(1-\delta_Q(p))
&\ll
\prod_{p\le X_0}
\left(1+(h(p)-1)\delta_Q(p)+O(p^{-2})\right)\notag\\
&\quad\times
\prod_{X_0<p\le v}
\left(1+h(p)\delta_Q(p)+O(p^{-2})\right),
\end{align}
with the finitely many primes $p\le g$ absorbed into the constant. Condition (B) gives
\[
(X^*)^{1/A}\le X_0\le X^*,
\qquad
X^*\le v\le mX^*.
\]
Since $h(p)$ is bounded and $\delta_Q(p)\ll_Qp^{-1}$,
\[
\sum_{X_0<p\le v}h(p)\delta_Q(p)\ll_{Q,h,A,m}1
\]
and likewise
\[
\sum_{X_0<p\le X^*}\abs{h(p)-1}\delta_Q(p)\ll_{Q,h,A}1.
\]
Thus both the unmatched range $X_0<p\le v$ and the replacement of $X_0$ by $X^*$ cost only a constant factor. Taking logarithms in the first product in \eqref{eq:cutoffcompare} and using $\sum_pp^{-2}<\infty$ yields
\[
E_Q(v)\prod_{g<p\le X_0}(1-\delta_Q(p))
\ll_{Q,h,A}
\exp\left\{\sum_{p\le X^*}(h(p)-1)\delta_Q(p)\right\}.
\]
Since $\delta_Q(p)=\omega_Q(p)/p$, substitution into \eqref{eq:BTapplication} proves \eqref{eq:unified}.
\end{proof}

\begin{remark}
For a global box, a closely related upper bound follows from \cite[Theorem 1.15]{CKPS2025}.
\end{remark}

\section{A unified joint-prime framework}\label{sec:CST}

The prime-level factor in the polynomial-value sieve is the product of two pieces of information: the top-weight permutation character $\tau_Q$ attached to the polynomial and the prime-level class function attached to the arithmetic weight. The latter may be automorphic, purely finite Galois, or a product of the two. We estimate the resulting prime sums using Chebotarev--Sato--Tate equidistribution.

\subsection{The Chebotarev--Sato--Tate group}

Let $M/\Q$ be a finite Galois extension containing $M_Q$, with
\[
G=\Gal(M/\Q).
\]
Inflate $\tau_Q$ from $G_Q$ to $G$. Let $S$ be the compact Sato--Tate group attached to the automorphic or motivic datum under consideration; $S$ is allowed to be trivial. At every prime outside a fixed finite set, denote the joint conjugacy class by
\[
j_p=(\Frob_p(M/\Q),s_p)\in G\times S.
\]
When no confusion can arise, we simply write $\Frob_p$ for the Frobenius class in $G$. The Chebotarev--Sato--Tate group considered here is a closed subgroup
\[
G_S\subseteq G\times S.
\]
We take $G_S$ with its natural surjective projections onto both $G$ and $S$. In many applications $G_S$ is therefore forced to be the direct product; the CM case, where the Sato--Tate group is disconnected, is an important exception.

Let $\mu_{G_S}$ denote normalized Haar measure on $G_S$.

\begin{hypothesis}[Chebotarev--Sato--Tate]\label{hyp:CST}
The conjugacy classes $j_p$ are equidistributed in the space of conjugacy classes of $G_S$ with respect to $\mu_{G_S}$. Equivalently, for every continuous class function $F:G_S\to\C$,
\begin{equation}\label{eq:CSTqual}
\frac1{\pi(x)}\sum_{p\le x}F(j_p)
\longrightarrow
\int_{G_S}F\,\dd\mu_{G_S}.
\end{equation}
\end{hypothesis}

\begin{remark}
Serre's equidistribution criterion \cite{Serre1981} reduces Hypothesis~\ref{hyp:CST} to analytic properties at $s=1$ of the Euler products attached to nontrivial irreducible characters of $G_S$.
\end{remark}

For a bounded nonnegative continuous class function $\Phi$ on $G_S$, define
\begin{equation}\label{eq:kappa}
\kappa_Q(\Phi)
:=\int_{G_S}\tau_Q(g)\Phi(g,s)\,\dd\mu_{G_S}(g,s).
\end{equation}

\begin{proposition}\label{prop:qualitativeMertens}
Under Hypothesis~\ref{hyp:CST},
\begin{equation}\label{eq:qualMertens}
\sum_{p\le x}\frac{\tau_Q(\Frob_p)\Phi(j_p)}p
=\kappa_Q(\Phi)\log\log x+o(\log\log x).
\end{equation}
\end{proposition}

\begin{proof}
Apply \eqref{eq:CSTqual} to the continuous class function $(g,s)\mapsto\tau_Q(g)\Phi(g,s)$. This gives
\[
\sum_{p\le x}\tau_Q(\Frob_p)\Phi(j_p)
=\kappa_Q(\Phi)\pi(x)+o(\pi(x)).
\]
Partial summation and the prime number theorem give \eqref{eq:qualMertens}.
\end{proof}

\subsection{A Mertens-effective joint distribution}
To obtain the precise logarithmic exponent, we also need an effective version of the Chebotarev--Sato--Tate distribution. For our purposes, it suffices to assume the following Mertens-effective form of Chebotarev--Sato--Tate equidistribution.

\begin{hypothesis}\label{hyp:effectiveCST}
Let $F:G_S\to\C$ be a continuous class function and put
\[
\mu(F)=\int_{G_S}F\,\dd\mu_{G_S}.
\]
We say that $F$ satisfies the effective Chebotarev--Sato--Tate hypothesis if
\begin{equation}\label{eq:effectiveCST}
\sum_{p\le x}F(j_p)=\mu(F)\pi(x)+E_F(x)
\end{equation}
with
\begin{equation}\label{eq:integrableError}
\int_3^\infty\frac{\abs{E_F(t)}}{t^2}\,\dd t<\infty.
\end{equation}
\end{hypothesis}

\begin{proposition}\label{prop:effectiveMertens}
If the test function $F_Q(g,s)=\tau_Q(g)\Phi(g,s)$ satisfies Hypothesis~\ref{hyp:effectiveCST}, then
\begin{equation}\label{eq:effectiveMertens}
\sum_{p\le x}\frac{\tau_Q(\Frob_p)\Phi(j_p)}p
=\kappa_Q(\Phi)\log\log x+O(1).
\end{equation}
\end{proposition}

\begin{proof}
Abel summation shows that the main term in \eqref{eq:effectiveCST} contributes $\kappa_Q(\Phi)\log\log x$ plus a constant. The contribution of $E_{F_Q}$ converges by \eqref{eq:integrableError}.
\end{proof}

If $G_S=G\times S$ and $\Phi(g,s)=W(g)H(s)$ factors, then
\begin{equation}\label{eq:kappafactor}
\kappa_Q(\Phi)=\langle\tau_Q,W\rangle_G\int_SH\,\dd\mu_S.
\end{equation}
In particular, if $W=1$, then
\begin{equation}\label{eq:kappaSTfactor}
\kappa_Q(H)=r(Q)\int_SH\,\dd\mu_S.
\end{equation}

Combining \cref{thm:unified} and the Chebotarev--Sato--Tate hypothesis, we obtain the following theorem.
\begin{theorem}\label{cor:CSTsieve}
Let $h$ be admissible and suppose that, at every good prime,
\[
h(p)=\Phi(j_p)
\]
for a bounded nonnegative continuous class function $\Phi$ on a Chebotarev--Sato--Tate group $G_S$. If the function
\[
F_Q(g,s):=\tau_Q(g)\Phi(g,s)
\]
satisfies Hypothesis~\ref{hyp:effectiveCST}, then
\begin{equation}\label{eq:CSTsieve}
\sum_{\substack{\mathbf n\in\mathcal B\\Q(\mathbf n)\ne0}}
h(\abs{Q(\mathbf n)})
\ll_{Q,h,A}V(\log X^*)^{\kappa_Q(\Phi)-r(Q)}.
\end{equation}
If only the qualitative Hypothesis~\ref{hyp:CST} is known, then for every $\varepsilon>0$ the exponent in \eqref{eq:CSTsieve} may be replaced by $\kappa_Q(\Phi)-r(Q)+\varepsilon$.
\end{theorem}

\begin{proof}
By \cref{lem:topweight}, boundedness of $\Phi$, and \cref{prop:effectiveMertens}, we have
\[
\sum_{p\le x}\frac{\omega_Q(p)h(p)}p
=\kappa_Q(\Phi)\log\log x+O(1).
\]
Together with \eqref{eq:omegaMertens}, \cref{thm:unified} gives \eqref{eq:CSTsieve}. In the qualitative case, \cref{prop:qualitativeMertens} gives an $o(\log\log X^*)$ error, which is absorbed into $(\log X^*)^\varepsilon$.
\end{proof}

\subsection{Cases in which the joint distribution is known}\label{subsec:knownCST}

We briefly record the cases relevant to the applications below in which the required joint distribution is known, and indicate when an effective form is available.
\subsubsection{Purely Galois case}

If $S$ is trivial, the framework reduces to ordinary Chebotarev. For a fixed Galois extension $M/\Q$ and any class function $\Phi$ on $G=\Gal(M/\Q)$, the effective Chebotarev density theorem \cite{LagariasOdlyzko1977,Serre1981} gives the Mertens formula
\begin{equation}\label{eq:pureGaloisMertens}
\sum_{p\le x}\frac{\tau_Q(\Frob_p)\Phi(\Frob_p)}p
=\langle\tau_Q,\Phi\rangle_G\log\log x+O(1).
\end{equation}
All purely finite Galois applications in this paper are therefore unconditional.

\subsubsection{Non-CM case}

Let $\pi$ be a cuspidal automorphic representation of $\mathrm{GL}_2(\A_{\Q})$ corresponding to a primitive non-CM holomorphic newform; for simplicity we take the central character to be trivial. The Sato--Tate group is $\mathrm{SU}(2)$. Writing
\[
\lambda_\pi(p)=2\cos\theta_p,
\qquad 0\le\theta_p\le\pi,
\]
the Sato--Tate measure is
\begin{equation}\label{eq:STmeasure}
\dd\mu_{\ST}(\theta)=\frac2\pi\sin^2\theta\,\dd\theta.
\end{equation}
For two twist-inequivalent non-CM newforms $\pi_1,\pi_2$ with trivial central characters, the relevant joint Sato--Tate group is $\mathrm{SU}(2)\times\mathrm{SU}(2)$ with product measure
\[
\dd\mu_{\ST}^{(2)}(\theta_1,\theta_2)
=\frac4{\pi^2}\sin^2\theta_1\sin^2\theta_2\,\dd\theta_1\dd\theta_2.
\]

The automorphy of all symmetric powers of non-CM holomorphic newforms over $\Q$ is due to Newton and Thorne \cite{NewtonThorneI,NewtonThorneII}; the corresponding Hilbert modular symmetric-power functoriality is now also known \cite{NewtonThorneHilbert}. These results provide basic tools to prove the Chebotarev--Sato--Tate theorems.

Murty and Murty \cite{MurtyMurty2009} first established a Chebotarev--Sato--Tate theorem when $M/\Q$ is solvable. Later, Wong \cite{Wong2019} developed a systematic treatment for non-CM Hilbert modular forms and proved that, if $G=\Gal(M/\Q)$ factors as $G=G_1\times G_2$ with $M^{G_1}$ totally real and $G_1$ abelian, then the Chebotarev--Sato--Tate theorem holds for $G\times \SU(2)$ and $G\times \SU(2)\times\SU(2)$.

For the effective setting, Thorner's results \cite{Thorner2021,Thorner2025} provide the estimates required by Hypothesis~\ref{hyp:effectiveCST} for the test functions occurring below when $M/\Q$ is abelian or totally real. Moreover, Wong \cite{Wong2019} proved an effective Chebotarev--Sato--Tate theorem under GRH.



\subsubsection{The CM case}

Let $\pi$ be a primitive CM holomorphic newform over $\Q$ with trivial central character and CM by an imaginary quadratic field $E$. Equivalently, $\pi$ is automorphically induced from an algebraic Hecke character of $E$. Let
\[
\chi_E:\GQ\longrightarrow\{\pm1\}
\]
be the quadratic character of $E/\Q$, and put
\[
T=\left\{
\begin{pmatrix}e^{i\theta}&0\\0&e^{-i\theta}\end{pmatrix}:\theta\in\mathbf R
\right\}\simeq \mathrm U(1).
\]
The Sato--Tate group is the normalizer
\[
S_{\CM}=N_{\mathrm{SU}(2)}(T)=T\sqcup wT,
\qquad
w=\begin{pmatrix}0&1\\-1&0\end{pmatrix},
\]
with component character
\[
c:S_{\CM}\longrightarrow\{\pm1\},
\qquad c(T)=1,\quad c(wT)=-1.
\]
Enlarge $M$ so that $E\subset M$, and continue to write $G=\Gal(M/\Q)$. The joint group is the fiber product
\begin{equation}\label{eq:CMfiber}
G_S^{\CM}
=\{(\sigma,s)\in G\times S_{\CM}:\chi_E(\sigma)=c(s)\}.
\end{equation}

In this case, the reciprocal-prime estimate needed here is also unconditional.
\begin{lemma}\label{lem:CMMertens}
For each fixed $d\ge1$,
\begin{equation}\label{eq:CMMertensFormula}
\sum_{p\le x}\frac{\tau_Q(\Frob_p)\abs{\lambda_{\Sym^d\pi}(p)}}p
=\kappa_Q(H_d)\log\log x+O_{Q,\pi,d}(1),
\end{equation}
where $H_d(s)=\abs{\tr(\Sym^d s)}$ and $\kappa_Q(H_d)$ is the Haar average on the fiber product $G_S^{\CM}$.
\end{lemma}

\begin{proof}
We sketch the standard Chebotarev--Hecke reduction, since only the reciprocal-prime form \eqref{eq:CMMertensFormula} is needed below. On the disconnected component $wT$, the function $H_d$ is constant, so its contribution is a purely finite Chebotarev prime sum and has the required form by \eqref{eq:pureGaloisMertens}.

On $T$, write $s=\operatorname{diag}(e^{i\theta},e^{-i\theta})$. The continuous piecewise smooth function
\[
\theta\longmapsto H_d(s)=\abs{U_d(\cos\theta)}
\]
has an absolutely convergent Fourier expansion
\[
H_d(s)=\sum_{n\in\mathbf Z}c_{d,n}e^{in\theta},
\qquad
c_{d,n}\ll_d(1+\abs n)^{-2}.
\]
After restricting the finite character $\tau_Q$ to the CM field and applying Brauer induction, each nonzero Fourier mode reduces to a finite integral combination of prime sums attached to Hecke characters over fixed intermediate number fields: the infinite-order part is a nonzero power of the unitary Hecke character inducing $\pi$, and the remaining factor has finite order. The corresponding Hecke $L$-functions are holomorphic and nonzero at $s=1$. Their Euler products therefore give bounded reciprocal-prime sums for every nonzero mode; standard bounds in terms of the archimedean conductor give $O_{Q,\pi,d}(\log(2+\abs n))$. Since
\[
\sum_{n\in\mathbf Z}\abs{c_{d,n}}\log(2+\abs n)<\infty,
\]
the sum of all nonzero modes is $O_{Q,\pi,d}(1)$. The zero Fourier mode contributes exactly the Haar average on the split component times $\log\log x$. Combining the split and inert components yields \eqref{eq:CMMertensFormula}. The underlying Hecke and Brauer--Chebotarev ingredients are classical; see \cite{Hecke1920,Serre1981}.
\end{proof}

\section{Applications to arithmetic functions}\label{sec:applications}

\subsection{Non-CM symmetric powers and Rankin--Selberg-type products}\label{subsec:nonCMapps}

Let $\pi$ be a cuspidal automorphic representation of $\mathrm{GL}_2(\A_{\Q})$ corresponding to a primitive non-CM holomorphic newform over $\Q$, with trivial central character for simplicity. For $d\ge1$, let $\Sym^d\pi$ be the $d$-th symmetric power representation. At every unramified prime $p$, write
\[
\lambda_\pi(p)=2\cos\theta_p,
\qquad 0\le\theta_p\le\pi.
\]
The irreducible characters of $\mathrm{SU}(2)$ are
\[
\chi_d(\theta)=U_d(\cos\theta)
=\frac{\sin((d+1)\theta)}{\sin\theta}.
\]
For the absolute symmetric-power prime weight
\[
H_d(\theta)=\abs{U_d(\cos\theta)},
\]
put
\begin{equation}\label{eq:kappad}
\kappa_d^{\nCM}
:=\frac2\pi\int_0^\pi\abs{U_d(\cos\theta)}\sin^2\theta\,\dd\theta.
\end{equation}
Equation \eqref{eq:kappaSTfactor} gives
\begin{equation}\label{eq:kappaQHd}
\kappa_Q(H_d)=r(Q)\kappa_d^{\nCM}.
\end{equation}

For positive integers $a,b$, define similarly
\begin{align}\label{eq:kappaab}
\kappa_{a,b}^{\Delta}
&:=\frac2\pi\int_0^\pi
\abs{U_a(\cos\theta)U_b(\cos\theta)}\sin^2\theta\,\dd\theta\notag\\
&=\frac2\pi\int_0^\pi
\abs{\sin((a+1)\theta)\sin((b+1)\theta)}\,\dd\theta.
\end{align}
For $H_{a,b}(\theta)=\abs{U_a(\cos\theta)U_b(\cos\theta)}$,
\begin{equation}\label{eq:kappaQHab}
\kappa_Q(H_{a,b})=r(Q)\kappa_{a,b}^{\Delta}.
\end{equation}
At every unramified prime,
\[
\lambda_{\Sym^d\pi}(p)=U_d(\cos\theta_p).
\]

\begin{theorem}\label{thm:nonCM}
Assume that the test functions associated with $(Q,\pi)$ satisfy Hypothesis~\ref{hyp:effectiveCST}. Then
\begin{equation}\label{eq:symdpoly}
\sum_{\substack{\mathbf n\in\mathcal B\\Q(\mathbf n)\ne0}}
\abs{\lambda_{\Sym^d\pi}(\abs{Q(\mathbf n)})}
\ll V(\log X^*)^{-r(Q)(1-\kappa_d^{\nCM})},
\end{equation}
and
\begin{equation}\label{eq:symaSymbpoly}
\sum_{\substack{\mathbf n\in\mathcal B\\Q(\mathbf n)\ne0}}
\abs{\lambda_{\Sym^a\pi}(\abs{Q(\mathbf n)})
\lambda_{\Sym^b\pi}(\abs{Q(\mathbf n)})}
\ll V(\log X^*)^{-r(Q)(1-\kappa_{a,b}^{\Delta})}.
\end{equation}
With only the qualitative joint distribution, the exponents on the right-hand sides may be increased by an arbitrary $\varepsilon>0$.
\end{theorem}

\begin{proof}
The local Euler factor of $\Sym^d\pi$ has degree $d+1$ with unitary Satake parameters, so
\[
\abs{\lambda_{\Sym^d\pi}(p^\nu)}\ll_d(\nu+1)^d.
\]
Thus $n\mapsto\abs{\lambda_{\Sym^d\pi}(n)}$ is admissible. Applying \cref{cor:CSTsieve} together with \eqref{eq:kappaQHd} gives \eqref{eq:symdpoly}. The product in \eqref{eq:symaSymbpoly} is also admissible, and \eqref{eq:kappaQHab} gives the second estimate.
\end{proof}

Let $\pi_1$ and $\pi_2$ be two cuspidal automorphic representations of $\mathrm{GL}_2(\A_{\Q})$ corresponding to primitive non-CM holomorphic newforms with trivial central characters. We next consider another Rankin--Selberg-type sum.

\begin{theorem}\label{thm:twoforms}
Assume that $\pi_1$ and $\pi_2$ are not twist-equivalent and that the test function associated with $(Q,\pi_1,\pi_2)$ satisfies Hypothesis~\ref{hyp:effectiveCST}. Then
\begin{equation}\label{eq:twoforms}
\sum_{\substack{\mathbf n\in\mathcal B\\Q(\mathbf n)\ne0}}
\abs{\lambda_{\Sym^a\pi_1}(\abs{Q(\mathbf n)})
\lambda_{\Sym^b\pi_2}(\abs{Q(\mathbf n)})}
\ll V(\log X^*)^{-r(Q)(1-\kappa_a^{\nCM}\kappa_b^{\nCM})}.
\end{equation}
\end{theorem}

\begin{proof}
For twist-inequivalent forms, the joint Sato--Tate group is $\mathrm{SU}(2)\times\mathrm{SU}(2)$ with product measure. Thus the prime-level mean factors as $\kappa_a^{\nCM}\kappa_b^{\nCM}$, and the result follows from \cref{cor:CSTsieve}.
\end{proof}

\subsection{CM cases}\label{subsec:CMapps}

Assume now that $\pi$ is a primitive CM holomorphic newform with trivial central character and CM by the imaginary quadratic field $E$. Recall the fiber-product group $G_S^{\CM}$ in \eqref{eq:CMfiber}. For a bounded class function $H$ on $S_{\CM}$, define the conditional component means
\[
H_+=\int_T H(t)\,\dd\mu_T(t),
\qquad
H_-=\int_{wT}H(t)\,\dd\mu_{wT}(t),
\]
and put
\begin{equation}\label{eq:aHbH}
a_H=\frac{H_++H_-}{2},
\qquad
b_H=\frac{H_+-H_-}{2}.
\end{equation}

\begin{proposition}\label{prop:CMconstant}
Let $H$ be a bounded class function on $S_{\CM}$. Then
\[
\kappa_Q(H)=a_H\langle\tau_Q,1\rangle_G+b_H\langle\tau_Q,\chi_E\rangle_G.
\]
Let
\[
\eta_Q(E):=\langle\tau_Q,\chi_E\rangle_G.
\]
Then we have
\begin{equation}\label{eq:CMkappa}
\kappa_Q(H)=r(Q)a_H+\eta_Q(E)b_H.
\end{equation}
\end{proposition}

\begin{proof}
For each $\sigma\in G$, the fiber-product condition forces the Sato--Tate variable to lie in the component indexed by $\chi_E(\sigma)$. Hence
\[
\kappa_Q(H)=\frac1{\abs G}\sum_{\sigma\in G}\tau_Q(\sigma)H_{\chi_E(\sigma)}.
\]
By \eqref{eq:aHbH},
\[
H_{\chi_E(\sigma)}=a_H+b_H\chi_E(\sigma).
\]
Taking the two finite-group inner products and using $\langle\tau_Q,1\rangle_G=r(Q)$ gives \eqref{eq:CMkappa}.
\end{proof}

If $Q$ is irreducible over $\Q$, then $\tau_Q=\Ind_H^G\1$ for the stabilizer $H$ of one geometric component. Frobenius reciprocity gives
\begin{equation}\label{eq:etaCMirreducible}
\eta_Q(E)=
\begin{cases}
1,&\chi_E\vert_H=1,\\
0,&\chi_E\vert_H\ne1.
\end{cases}
\end{equation}
Thus $\eta_Q(E)=1$ precisely when the CM field is contained in the field of definition of that geometric component.

For
\[
H_d(s)=\abs{\tr(\Sym^d s)},
\]
set
\begin{equation}\label{eq:AdBd}
A_d=\frac1\pi\int_0^\pi\abs{U_d(\cos\theta)}\,\dd\theta,
\qquad
B_d=\abs{U_d(0)}=
\begin{cases}
0,&d\text{ odd},\\
1,&d\text{ even}.
\end{cases}
\end{equation}
Then $H_+=A_d$ and $H_-=B_d$, and \cref{prop:CMconstant} yields
\begin{equation}\label{eq:CMHdconstant}
\kappa_Q(H_d)
=\frac{r(Q)+\eta_Q(E)}2A_d
+\frac{r(Q)-\eta_Q(E)}2B_d.
\end{equation}

\begin{theorem}\label{thm:CM}
Let $\pi$ be as above. Then
\begin{equation}\label{eq:CMbound}
\sum_{\substack{\mathbf n\in\mathcal B\\Q(\mathbf n)\ne0}}
\abs{\lambda_{\Sym^d\pi}(\abs{Q(\mathbf n)})}
\ll V(\log X^*)^{\kappa_Q(H_d)-r(Q)},
\end{equation}
where $\kappa_Q(H_d)$ is given by \eqref{eq:CMHdconstant}. If $Q$ is irreducible, then $r(Q)=1$ and $\eta_Q(E)\in\{0,1\}$.
\end{theorem}

\begin{proof}
The function $n\mapsto\abs{\lambda_{\Sym^d\pi}(n)}$ is admissible, and the theorem follows from \cref{cor:CSTsieve}.
\end{proof}

\subsection{Dedekind zeta-function coefficients}

Let $K/\Q$ be a number field. Enlarge $M$ if necessary so that it contains both $M_Q$ and the Galois closure of $K$. Put
\[
H_K=\Gal(M/K),
\qquad
\tau_K=\Ind_{H_K}^G\1.
\]
This is the permutation character on the embeddings of $K$. At every unramified prime,
\[
a_K(p)=\tau_K(\Frob_p),
\]
where
\[
\zeta_K(s)=\sum_{n\ge1}\frac{a_K(n)}{n^s}.
\]
Taking $\Phi=\tau_K$ in \eqref{eq:kappa}, we obtain
\begin{equation}\label{eq:kappaQK}
\kappa_Q(\tau_K)=\langle\tau_Q,\tau_K\rangle_G=: \kappa_Q(K).
\end{equation}
The effective Chebotarev density theorem gives
\begin{equation}\label{eq:DedekindMertens}
\sum_{p\le x}\frac{\tau_Q(\Frob_p)a_K(p)}p
=\kappa_Q(K)\log\log x+O(1).
\end{equation}

The pairing has a concrete orbit interpretation. If $\tau_Q$ is transitive, say $\tau_Q=\Ind_H^G\1$, then Mackey theory gives
\begin{equation}\label{eq:doublecoset}
\kappa_Q(K)=\#(H\backslash G/H_K).
\end{equation}
If $Q$ is irreducible over $\Q$, then $\tau_Q$ is transitive, and linear disjointness of $M_Q$ from the Galois closure of $K$ gives $\kappa_Q(K)=1$. More generally, under the same linear-disjointness condition one has
\[
\kappa_Q(K)=r(Q).
\]
Indeed, on the compositum the finite Galois group is the direct product of the two relevant Galois groups, so the scalar product factors as $\langle\tau_Q,1\rangle\langle\tau_K,1\rangle=r(Q)$.

\begin{theorem}\label{thm:Dedekind}
We have
\begin{equation}\label{eq:Dedekindbound}
\sum_{\substack{\mathbf n\in\mathcal B\\Q(\mathbf n)\ne0}}
a_K(\abs{Q(\mathbf n)})
\ll_{Q,K,A}V(\log X^*)^{\kappa_Q(K)-r(Q)}.
\end{equation}
\end{theorem}

\begin{proof}
Equation \eqref{eq:DedekindMertens} supplies the required prime formula. Moreover,
\[
a_K(p^\nu)\le
\binom{\nu+[K:\Q]-1}{[K:\Q]-1},
\]
so $a_K$ is admissible. The theorem follows from \cref{cor:CSTsieve}.
\end{proof}

\subsection{The mixed case}

The class-function formulation also permits finite and automorphic weights to be multiplied without changing the sieve argument. Let $K/\Q$ be a number field, let $M/\Q$ be a finite Galois extension containing $M_Q$ and the Galois closure of $K$, and let $\pi$ correspond to a primitive non-CM holomorphic newform over $\Q$ with trivial central character. With $G=\Gal(M/\Q)$, define $\tau_K$ as in the preceding subsection and take
\[
\Phi(g,s)=\tau_K(g)\abs{\tr(\Sym^d s)}.
\]
At good primes,
\[
\Phi(j_p)=a_K(p)\abs{\lambda_{\Sym^d\pi}(p)}.
\]
The relevant logarithmic constant is
\begin{equation}\label{eq:mixedconstant}
\kappa_Q(K,\pi,d)
=\int_{G_S}\tau_Q(g)\tau_K(g)\abs{\tr(\Sym^d s)}\,\dd\mu_{G_S}(g,s).
\end{equation}
Since $\mathrm{SU}(2)$ is connected and $G$ is finite, the joint group is $G_S=G\times\mathrm{SU}(2)$, and therefore
\begin{equation}\label{eq:mixedfactor}
\kappa_Q(K,\pi,d)=\kappa_Q(K)\kappa_d^{\nCM}.
\end{equation}

\begin{theorem}\label{thm:mixed}
Let $K$ and $\pi$ be as above, and assume that the test function associated with $(Q,K,\pi)$ satisfies Hypothesis~\ref{hyp:effectiveCST}. Then
\begin{equation}\label{eq:mixedbound}
\sum_{\substack{\mathbf n\in\mathcal B\\Q(\mathbf n)\ne0}}
a_K(\abs{Q(\mathbf n)})
\abs{\lambda_{\Sym^d\pi}(\abs{Q(\mathbf n)})}
\ll V(\log X^*)^{\kappa_Q(K)\kappa_d^{\nCM}-r(Q)}.
\end{equation}
\end{theorem}

\begin{proof}
The bound
\[
a_K(p^\nu)\ll_K(\nu+1)^{[K:\Q]-1}
\]
and temperedness of the holomorphic newform give
\[
\abs{\lambda_{\Sym^d\pi}(p^\nu)}\ll_d(\nu+1)^d.
\]
Hence the product of the two arithmetic weights is admissible. Applying \cref{cor:CSTsieve} and \eqref{eq:mixedfactor} finishes the proof.
\end{proof}

\section{Applications to congruences of modular forms}\label{sec:congruences}

We now give an arithmetic application to congruences of Fourier coefficients of modular forms. It may be viewed as a polynomial-value analogue of the classical divisibility phenomena studied by Serre \cite{Serre1976}.

Let
\[
f(z)=\sum_{n\ge1}a_f(n)q^n
\]
be a normalized cuspidal Hecke eigenform of weight $k\ge2$, level $N$, and nebentypus $\chi$, and let $E_f$ be its coefficient field. Fix a prime ideal
\[
\lambda\subset\mathcal O_{E_f}
\]
above a rational prime $\ell$, and write
\[
k_\lambda=\mathcal O_{E_f}/\lambda.
\]
Associated with $f$ and $\lambda$ is a residual Galois representation
\[
\bar\rho_{f,\lambda}:\GQ\longrightarrow\mathrm{GL}_2(k_\lambda)
\]
such that, for every prime $p\nmid N\ell$,
\begin{equation}\label{eq:residualtrace}
\tr\bar\rho_{f,\lambda}(\Frob_p)\equiv a_f(p)\pmod\lambda.
\end{equation}

\subsection{Divisibility of Fourier coefficients along polynomial values}

Let $Q$ be a primitive square-free polynomial in $\Z[x_1,\ldots,x_m]$, and let $M_Q$ be the component field introduced in \cref{sec:local}. Let $L_{f,\lambda}$ be the finite Galois extension fixed by $\ker\bar\rho_{f,\lambda}$, and put
\[
M=M_QL_{f,\lambda},
\qquad
G=\Gal(M/\Q).
\]
Both $\tau_Q$ and $\bar\rho_{f,\lambda}$ are regarded as representations of $G$. Define the trace-zero class function
\begin{equation}\label{eq:zflambda}
z_{f,\lambda}(\sigma)=
\begin{cases}
1,&\tr\bar\rho_{f,\lambda}(\sigma)=0\text{ in }k_\lambda,\\
0,&\tr\bar\rho_{f,\lambda}(\sigma)\ne0.
\end{cases}
\end{equation}
Set
\begin{equation}\label{eq:deltageneral}
\delta_{Q,f,\lambda}
:=\langle\tau_Q,z_{f,\lambda}\rangle_G.
\end{equation}

\begin{theorem}\label{thm:divisibility}
Let $Q\in\Z[x_1,\ldots,x_m]$ be primitive and square-free over $\Q$, and let $\mathcal B$ satisfy (B). Then
\begin{equation}\label{eq:divisibility}
\#\{\mathbf n\in\mathcal B:Q(\mathbf n)\ne0,\ \lambda\nmid a_f(\abs{Q(\mathbf n)})\}
\ll_{Q,f,\lambda,A}
\frac{V}{(\log X^*)^{\delta_{Q,f,\lambda}}}.
\end{equation}
In particular, whenever $\delta_{Q,f,\lambda}>0$, the coefficient $a_f(\abs{Q(\mathbf n)})$ is divisible by $\lambda$ for all but a logarithmically sparse set of polynomial parameters.
\end{theorem}

\begin{proof}
Define
\[
b_{f,\lambda}(n)=\1_{\{\lambda\nmid a_f(n)\}}.
\]
Since $a_f$ is multiplicative, $b_{f,\lambda}$ is multiplicative; moreover $0\le b_{f,\lambda}(p^\nu)\le1$, so it is admissible. For every prime $p$ outside a fixed finite set, \eqref{eq:residualtrace} gives
\[
b_{f,\lambda}(p)=1-z_{f,\lambda}(\Frob_p).
\]
This is a purely finite Galois weight. By \eqref{eq:pureGaloisMertens},
\[
\kappa_Q(1-z_{f,\lambda})
=\langle\tau_Q,1-z_{f,\lambda}\rangle_G
=r(Q)-\delta_{Q,f,\lambda}.
\]
Applying \cref{cor:CSTsieve} proves \eqref{eq:divisibility}.
\end{proof}

In the linearly disjoint case, the constant has a simple form. Suppose that
\begin{equation}\label{eq:LDresidual}
M_Q\cap L_{f,\lambda}=\Q.
\end{equation}
Denote
\[
G_{f,\lambda}=\im\bar\rho_{f,\lambda}.
\]
Then $G\simeq G_Q\times G_{f,\lambda}$, and hence
\begin{equation}\label{eq:deltaLD}
\delta_{Q,f,\lambda}=r(Q)\vartheta_{f,\lambda},
\qquad
\vartheta_{f,\lambda}:=
\frac{\#\{g\in G_{f,\lambda}:\tr g=0\}}{\abs{G_{f,\lambda}}}.
\end{equation}

\subsection{Full residual image case}

Suppose now that
\[
G_{f,\lambda}=\mathrm{GL}_2(\F_q),
\qquad q=\#k_\lambda.
\]
For every finite field $\F_q$,
\begin{equation}\label{eq:tracezerocount}
\#\{g\in\mathrm{GL}_2(\F_q):\tr g=0\}=q^2(q-1).
\end{equation}
Indeed, a trace-zero matrix has the form
\[
\begin{pmatrix}a&b\\c&-a\end{pmatrix}.
\]
There are $q^3$ such matrices, and the singular ones are given by $a^2+bc=0$. If $b\ne0$, then $a$ and $b$ determine $c$, giving $q(q-1)$ solutions; if $b=0$, then $a=0$ and $c$ is arbitrary, giving $q$ more. Thus there are $q^2$ singular trace-zero matrices, proving \eqref{eq:tracezerocount}. Consequently,
\begin{equation}\label{eq:fullimagedensity}
\frac{\#\{g\in\mathrm{GL}_2(\F_q):\tr g=0\}}{\abs{\mathrm{GL}_2(\F_q)}}
=\frac{q}{q^2-1}.
\end{equation}

\begin{corollary}\label{cor:fullimage}
Assume \eqref{eq:LDresidual} and
\[
\im\bar\rho_{f,\lambda}=\mathrm{GL}_2(\F_q).
\]
Then
\[
\#\{\mathbf n\in\mathcal B:Q(\mathbf n)\ne0,\ \lambda\nmid a_f(\abs{Q(\mathbf n)})\}
\ll_{Q,f,\lambda,A}
\frac{V}{(\log X^*)^{r(Q)q/(q^2-1)}}.
\]
\end{corollary}

Let $E/\Q$ be an elliptic curve and, for a prime $\ell$, let
\[
\bar\rho_{E,\ell}:\GQ\longrightarrow\mathrm{GL}_2(\F_\ell)
\]
be the representation on $E[\ell]$.

\begin{corollary}\label{cor:elliptic}
Assume that $\bar\rho_{E,\ell}$ is surjective and that the $\ell$-division field $\Q(E[\ell])$ is linearly disjoint from $M_Q$. Then
\[
\#\{\mathbf n\in\mathcal B:Q(\mathbf n)\ne0,\ \ell\nmid a_E(\abs{Q(\mathbf n)})\}
\ll_{Q,E,\ell,A}
\frac{V}{(\log X^*)^{r(Q)\ell/(\ell^2-1)}}.
\]
\end{corollary}

\subsection{Eisenstein congruence case}\label{subsec:Eisenstein}

We now consider the case in which the residual representation attached to $f$ is reducible. Assume that $\ell$ is odd and that $f$ is congruent modulo $\lambda$ to an Eisenstein eigensystem in the sense that, for all but finitely many primes $p$,
\begin{equation}\label{eq:Eisensteincong}
a_f(p)\equiv\psi_1(p)+\psi_2(p)p^{k-1}\pmod\lambda,
\end{equation}
where $\psi_1$ and $\psi_2$ are finite-order characters. Put
\[
\chi_1=\bar\psi_1,
\qquad
\chi_2=\bar\psi_2\bar\chi_\ell^{\,k-1},
\]
where $\bar\chi_\ell$ denotes the mod-$\ell$ cyclotomic character. Then
\[
\bar\rho_{f,\lambda}^{\mathrm{ss}}\simeq\chi_1\oplus\chi_2.
\]
After a choice of basis, one may write
\[
\bar\rho_{f,\lambda}(\sigma)=
\begin{pmatrix}
\chi_1(\sigma)&*\\0&\chi_2(\sigma)
\end{pmatrix}.
\]
Define the ratio character
\begin{equation}\label{eq:xidef}
\xi=\chi_2\chi_1^{-1}
=\bar\psi_2\bar\psi_1^{-1}\bar\chi_\ell^{\,k-1}.
\end{equation}
Since $\chi_1(\sigma)\ne0$,
\begin{equation}\label{eq:tracezeroXi}
\tr\bar\rho_{f,\lambda}(\sigma)=0
\quad\Longleftrightarrow\quad
\xi(\sigma)=-1.
\end{equation}
Let
\[
L_\xi=\overline{\Q}^{\ker\xi},
\qquad
\Gamma_\xi=\Gal(L_\xi/\Q)\simeq\im\xi.
\]
Since $\im\xi$ is a finite subgroup of $k_\lambda^\times$, it is cyclic. Write
\[
m_\xi=\abs{\im\xi}.
\]
Put
\[
M_\xi=M_QL_\xi,
\qquad
G_\xi=\Gal(M_\xi/\Q).
\]
Both $\tau_Q$ and the class function $\1_{\{\xi=-1\}}$ factor through $G_\xi$. Thus their scalar product computed in the larger residual group $\Gal(M_QL_{f,\lambda}/\Q)$ is the same as the scalar product computed on $G_\xi$.

Because $\ell$ is odd, $\im\xi$ contains $-1$ if and only if $m_\xi$ is even. Hence the trace-zero density in the Eisenstein quotient is
\begin{equation}\label{eq:thetaXi}
\vartheta_\xi
:=\frac{\#\{\gamma\in\Gamma_\xi:\xi(\gamma)=-1\}}{\abs{\Gamma_\xi}}
=
\begin{cases}
1/m_\xi,&m_\xi\text{ even},\\
0,&m_\xi\text{ odd}.
\end{cases}
\end{equation}

\begin{proposition}\label{prop:Eisenstein}
Retain the notation above. Then
\begin{equation}\label{eq:Eisensteinbound}
\#\{\mathbf n\in\mathcal B:Q(\mathbf n)\ne0,\ \lambda\nmid a_f(\abs{Q(\mathbf n)})\}
\ll_{Q,f,\lambda,A}
\frac{V}{(\log X^*)^{\delta_{Q,f,\lambda}}},
\end{equation}
where
\begin{equation}\label{eq:deltaXi}
\delta_{Q,f,\lambda}
=\left\langle\tau_Q,\1_{\{\xi=-1\}}\right\rangle_{G_\xi}.
\end{equation}
If moreover
\begin{equation}\label{eq:LDxi}
M_Q\cap L_\xi=\Q,
\end{equation}
then
\begin{equation}\label{eq:deltaXiLD}
\delta_{Q,f,\lambda}=r(Q)\vartheta_\xi
=
\begin{cases}
r(Q)/m_\xi,&m_\xi\text{ even},\\
0,&m_\xi\text{ odd}.
\end{cases}
\end{equation}
\end{proposition}

\begin{proof}
At every prime outside a fixed finite set, \eqref{eq:tracezeroXi} gives
\[
\lambda\mid a_f(p)
\quad\Longleftrightarrow\quad
\tr\bar\rho_{f,\lambda}(\Frob_p)=0
\quad\Longleftrightarrow\quad
\xi(\Frob_p)=-1.
\]
Thus the trace-zero class function in \cref{thm:divisibility} is exactly $\1_{\{\xi=-1\}}$, proving \eqref{eq:Eisensteinbound} and \eqref{eq:deltaXi}. Under \eqref{eq:LDxi}, we have $G_\xi\simeq G_Q\times\Gamma_\xi$, so the pairing factors:
\[
\delta_{Q,f,\lambda}
=\langle\tau_Q,1\rangle
\frac{\#\{\gamma\in\Gamma_\xi:\xi(\gamma)=-1\}}{\abs{\Gamma_\xi}}
=r(Q)\vartheta_\xi.
\]
Equation \eqref{eq:deltaXiLD} follows from \eqref{eq:thetaXi}.
\end{proof}

The constant can be made completely explicit even when the polynomial component field and the Eisenstein field are entangled. Write
\[
Q=R_1\cdots R_{r(Q)}
\]
as a product of distinct irreducible polynomials over $\Q$. For each $j$, choose a geometric irreducible component of $V(R_j)$ and let $H_j\subseteq G_\xi$ be its stabilizer. Then
\[
\tau_{R_j}=\Ind_{H_j}^{G_\xi}\1.
\]
Set
\[
m_j=\abs{\im(\xi\vert_{H_j})}.
\]

\begin{corollary}\label{cor:entangledEisenstein}
With the notation above,
\begin{equation}\label{eq:entangledformula}
\delta_{Q,f,\lambda}
=\sum_{\substack{1\le j\le r(Q)\\m_j\ \mathrm{even}}}\frac1{m_j}.
\end{equation}
\end{corollary}

\begin{proof}
Since $\1_{\{\xi=-1\}}$ is a class function, Frobenius reciprocity gives
\[
\left\langle\tau_{R_j},\1_{\{\xi=-1\}}\right\rangle_{G_\xi}
=\frac1{\abs{H_j}}\#\{h\in H_j:\xi(h)=-1\}.
\]
The image of $\xi\vert_{H_j}$ is cyclic of order $m_j$. Hence the last quantity is $1/m_j$ when $m_j$ is even and $0$ when $m_j$ is odd. Summing over $j$ gives \eqref{eq:entangledformula}.
\end{proof}

\subsubsection{Ramanujan's congruence modulo $691$}

For the discriminant form
\[
\Delta(z)=\sum_{n\ge1}\tau(n)q^n,
\]
Ramanujan's classical congruence gives
\[
\tau(p)\equiv1+p^{11}\pmod{691}.
\]
Hence
\[
\bar\rho_{\Delta,691}^{\mathrm{ss}}
\simeq1\oplus\bar\chi_{691}^{11},
\]
and the relevant ratio character is
\[
\xi=\bar\chi_{691}^{11}.
\]
Since $\gcd(11,690)=1$,
\[
m_\xi=690.
\]
Moreover, the map $x\mapsto x^{11}$ is an automorphism of $\F_{691}^\times$ and fixes $-1$, so
\[
\tr\bar\rho_{\Delta,691}(\Frob_p)=0
\quad\Longleftrightarrow\quad
p\equiv-1\pmod{691}.
\]
Thus $\vartheta_\xi=1/690$.

\begin{corollary}\label{cor:Ramanujan691}
If
\[
M_Q\cap\Q(\zeta_{691})=\Q,
\]
then
\[
\delta_{Q,\Delta,691}=\frac{r(Q)}{690}
\]
and
\begin{equation}\label{eq:Ramanujanbound}
\#\{\mathbf n\in\mathcal B:Q(\mathbf n)\ne0,\ 691\nmid\tau(\abs{Q(\mathbf n)})\}
\ll
\frac{V}{(\log X^*)^{r(Q)/690}}.
\end{equation}
Without the linear-disjointness assumption, let $G_\xi=\Gal(M_Q\Q(\zeta_{691})/\Q)$ as above; if $H_j\subseteq G_\xi$ is the stabilizer of a geometric component of $V(R_j)$ and
\[
m_j=\abs{\im(\bar\chi_{691}^{11}\vert_{H_j})},
\]
then
\[
\delta_{Q,\Delta,691}
=\sum_{\substack{1\le j\le r(Q)\\m_j\ \mathrm{even}}}\frac1{m_j}.
\]
\end{corollary}

\subsubsection{A colored-partition example}

A concrete example is provided by the weight-two eta-product
\[
f_{11}(z)=\eta(z)^2\eta(11z)^2\in S_2(\Gamma_0(11)).
\]
Write
\[
f_{11}(z)=\sum_{n\ge1}a_{11}(n)q^n.
\]
Define integers $v(n)$ by
\[
\prod_{m\ge1}(1-q^m)^2(1-q^{11m})^2
=\sum_{n\ge0}v(n)q^n.
\]
Since
\[
\eta(z)^2\eta(11z)^2
=q\prod_{m\ge1}(1-q^m)^2(1-q^{11m})^2,
\]
we have the exact identity
\begin{equation}\label{eq:a11v}
a_{11}(n)=v(n-1).
\end{equation}

The sequence $v(n)$ has a simple combinatorial interpretation. Let $v_e(n)$, respectively $v_o(n)$, denote the number of four-colored partitions of $n$ into distinct colored parts (that is, each colored part may occur at most once) with an even, respectively odd, total number of parts, where parts of the last two colors are required to be divisible by $11$. Then
\begin{equation}\label{eq:vevo}
v(n)=v_e(n)-v_o(n).
\end{equation}

The newform $f_{11}$ satisfies
\begin{equation}\label{eq:a11cong}
a_{11}(p)\equiv p+1\pmod5
\qquad(p\nmid55),
\end{equation}
which also follows from the fact that the corresponding elliptic curve has a rational point of order $5$; see, for example, \cite{AlacaAlacaAygin2018}. Equivalently,
\[
\bar\rho_{f_{11},5}^{\mathrm{ss}}\simeq1\oplus\bar\chi_5.
\]
Consequently,
\[
5\mid a_{11}(p)
\quad\Longleftrightarrow\quad
p\equiv-1\pmod5
\qquad(p\nmid55).
\]

\begin{corollary}\label{cor:colored}
Assume that
\[
M_Q\cap\Q(\zeta_5)=\Q.
\]
Then
\begin{equation}\label{eq:coloredbound}
\#\{\mathbf n\in\mathcal B:Q(\mathbf n)\ne0,\ 5\nmid v(\abs{Q(\mathbf n)}-1)\}
\ll_Q
\frac{V}{(\log X^*)^{r(Q)/4}}.
\end{equation}
Consequently, as $X^*\to\infty$ through boxes satisfying (B),
\[
v_e(\abs{Q(\mathbf n)}-1)
\equiv
v_o(\abs{Q(\mathbf n)}-1)
\pmod5
\]
for a density-one set of polynomial parameters $\mathbf n\in\mathcal B$ with $Q(\mathbf n)\ne0$.
\end{corollary}

\begin{proof}
The class $p\equiv-1\pmod5$ has density $1/4$ in $\Gal(\Q(\zeta_5)/\Q)\simeq(\Z/5\Z)^\times$. Thus \cref{prop:Eisenstein} gives $\delta_{Q,f_{11},5}=r(Q)/4$. Combining this with \eqref{eq:a11v} and \eqref{eq:vevo} gives \eqref{eq:coloredbound} and the final assertion.
\end{proof}

\bibliographystyle{plain}
\bibliography{references}

@article{AlacaAlacaAygin2018,
  author  = {Alaca, A. and Alaca, {\c{S}}. and Aygin, Z. S.},
  title   = {Eta quotients, {Eisenstein} series and elliptic curves},
  journal = {Integers},
  volume  = {18},
  year    = {2018},
  pages   = {A85},
  note    = {12 pp.}
}

@article{BDBrowning2006,
  author  = {de la Bret{\`e}che, R{\'e}gis and Browning, Timothy D.},
  title   = {Sums of arithmetic functions over values of binary forms},
  journal = {Acta Arith.},
  volume  = {125},
  number  = {3},
  year    = {2006},
  pages   = {291--304},
  doi     = {10.4064/aa125-3-6}
}

@article{BT2026,
  author  = {de la Bret{\`e}che, R{\'e}gis and Tenenbaum, G{\'e}rald},
  title   = {Mean values of arithmetic functions and application to sums of powers},
  journal = {Math. Proc. Cambridge Philos. Soc.},
  volume  = {180},
  number  = {1},
  year    = {2026},
  pages   = {1--13},
  doi     = {10.1017/S0305004125101382}
}

@article{CKPS2025,
  author        = {Chan, Stephanie and Koymans, Peter and Pagano, Carlo and Sofos, Efthymios},
  title         = {{$6$-torsion and integral points on quartic threefolds}},
  journal       = {Ann. Sc. Norm. Super. Pisa Cl. Sci.},
  year          = {2025},
  note          = {To appear},
  doi           = {10.2422/2036-2145.202412_006},
  eprint        = {2403.13359},
  archivePrefix = {arXiv},
  primaryClass  = {math.NT}
}

@article{CKPSAverages2025,
  author  = {Chan, Stephanie and Koymans, Peter and Pagano, Carlo and Sofos, Efthymios},
  title   = {Averages of multiplicative functions along equidistributed sequences},
  journal = {J. Number Theory},
  volume  = {273},
  year    = {2025},
  pages   = {1--36},
  doi     = {10.1016/j.jnt.2025.01.005}
}

@article{ChiriacYang2022,
  author  = {Chiriac, Liubomir and Yang, Liyang},
  title   = {Summing {Hecke} eigenvalues over polynomials},
  journal = {Math. Z.},
  volume  = {302},
  number  = {2},
  year    = {2022},
  pages   = {643--662},
  doi     = {10.1007/s00209-022-03071-y}
}

@article{deLaBretecheTenenbaum2012,
  author  = {de la Bret{\`e}che, R{\'e}gis and Tenenbaum, G{\'e}rald},
  title   = {Moyennes de fonctions arithm{\'e}tiques de formes binaires},
  journal = {Mathematika},
  volume  = {58},
  number  = {2},
  year    = {2012},
  pages   = {290--304},
  doi     = {10.1112/S0025579311002154}
}

@article{Erdos1952,
  author  = {Erd{\H{o}}s, Paul},
  title   = {On the sum $\sum_{k=1}^{x} d(f(k))$},
  journal = {J. London Math. Soc.},
  volume  = {27},
  year    = {1952},
  pages   = {7--15},
  doi     = {10.1112/jlms/s1-27.1.7}
}

@article{Fomenko1997,
  author  = {Fomenko, O. M.},
  title   = {Distribution of lattice points on surfaces of second order},
  journal = {J. Math. Sci.},
  volume  = {83},
  number  = {6},
  year    = {1997},
  pages   = {795--815},
  doi     = {10.1007/BF02439206},
  note    = {Russian original: Zap. Nauchn. Sem. POMI 212 (1994), 164--195}
}

@article{Hecke1920,
  author  = {Hecke, Erich},
  title   = {Eine neue Art von Zetafunktionen und ihre Beziehungen zur Verteilung der Primzahlen. Zweite Mitteilung},
  journal = {Math. Z.},
  volume  = {6},
  year    = {1920},
  pages   = {11--51}
}

@article{Kim2007,
  author  = {Kim, Henry H.},
  title   = {Functoriality and number of solutions of congruences},
  journal = {Acta Arith.},
  volume  = {128},
  number  = {3},
  year    = {2007},
  pages   = {235--243},
  doi     = {10.4064/aa128-3-4}
}

@incollection{LagariasOdlyzko1977,
  author    = {Lagarias, Jeffrey C. and Odlyzko, Andrew M.},
  title     = {Effective versions of the {Chebotarev} density theorem},
  booktitle = {Algebraic Number Fields: {$L$}-Functions and Galois Properties},
  editor    = {Fr{\"o}hlich, Armand},
  publisher = {Academic Press},
  address   = {London},
  year      = {1977},
  pages     = {409--464}
}

@article{LangWeil1954,
  author  = {Lang, Serge and Weil, Andr{\'e}},
  title   = {Number of points of varieties in finite fields},
  journal = {Amer. J. Math.},
  volume  = {76},
  year    = {1954},
  pages   = {819--827}
}

@article{LuoLao2024,
  author  = {Luo, Shu and Lao, Huixue},
  title   = {On the coefficients of automorphic representations over polynomials},
  journal = {Ramanujan J.},
  volume  = {65},
  number  = {1},
  year    = {2024},
  pages   = {173--188},
  doi     = {10.1007/s11139-024-00889-4}
}

@incollection{MurtyMurty2009,
  author    = {Murty, M. Ram and Murty, V. Kumar},
  title     = {The {Sato--Tate} conjecture and generalizations},
  booktitle = {Current Trends in Science, Platinum Jubilee Special},
  publisher = {Indian Academy of Sciences},
  address   = {Bangalore},
  year      = {2009},
  pages     = {639--646}
}

@article{Nair1992,
  author  = {Nair, Mohan},
  title   = {Multiplicative functions of polynomial values in short intervals},
  journal = {Acta Arith.},
  volume  = {62},
  number  = {3},
  year    = {1992},
  pages   = {257--269},
  doi     = {10.4064/aa-62-3-257-269}
}

@article{NairTenenbaum1998,
  author  = {Nair, Mohan and Tenenbaum, G{\'e}rald},
  title   = {Short sums of certain arithmetic functions},
  journal = {Acta Math.},
  volume  = {180},
  number  = {1},
  year    = {1998},
  pages   = {119--144},
  doi     = {10.1007/BF02392880}
}

@article{NewtonThorneI,
  author  = {Newton, James and Thorne, Jack A.},
  title   = {Symmetric power functoriality for holomorphic modular forms},
  journal = {Publ. Math. Inst. Hautes {\'E}tudes Sci.},
  volume  = {134},
  year    = {2021},
  pages   = {1--116},
  doi     = {10.1007/s10240-021-00127-3}
}

@article{NewtonThorneII,
  author  = {Newton, James and Thorne, Jack A.},
  title   = {Symmetric power functoriality for holomorphic modular forms, {II}},
  journal = {Publ. Math. Inst. Hautes {\'E}tudes Sci.},
  volume  = {134},
  year    = {2021},
  pages   = {117--152},
  doi     = {10.1007/s10240-021-00126-4}
}

@article{NewtonThorneHilbert,
  author  = {Newton, James and Thorne, Jack A.},
  title   = {Symmetric power functoriality for {Hilbert} modular forms},
  journal = {Ann. of Math. (2)},
  volume  = {203},
  number  = {1},
  year    = {2026},
  pages   = {283--347},
  doi     = {10.4007/annals.2026.203.1.4}
}

@article{Serre1976,
  author  = {Serre, Jean-Pierre},
  title   = {Divisibilit{\'e} de certaines fonctions arithm{\'e}tiques},
  journal = {Enseign. Math. (2)},
  volume  = {22},
  number  = {3--4},
  year    = {1976},
  pages   = {227--260}
}

@article{Serre1981,
  author  = {Serre, Jean-Pierre},
  title   = {Quelques applications du th{\'e}or{\`e}me de densit{\'e} de {Chebotarev}},
  journal = {Publ. Math. Inst. Hautes {\'E}tudes Sci.},
  volume  = {54},
  year    = {1981},
  pages   = {123--201}
}

@book{SerreNXp,
  author    = {Serre, Jean-Pierre},
  title     = {Lectures on {$N_X(p)$}},
  series    = {Research Notes in Mathematics},
  volume    = {11},
  publisher = {CRC Press},
  address   = {Boca Raton, FL},
  year      = {2012}
}

@article{Shiu1980,
  author  = {Shiu, P.},
  title   = {A {Brun--Titchmarsh} theorem for multiplicative functions},
  journal = {J. Reine Angew. Math.},
  volume  = {313},
  year    = {1980},
  pages   = {161--170},
  doi     = {10.1515/crll.1980.313.161}
}

@article{Thorner2021,
  author  = {Thorner, Jesse},
  title   = {Effective forms of the {Sato--Tate} conjecture},
  journal = {Res. Math. Sci.},
  volume  = {8},
  number  = {1},
  year    = {2021},
  pages   = {Paper No. 4},
  note    = {21 pp.},
  doi     = {10.1007/s40687-020-00234-3}
}

@article{Thorner2025,
  author  = {Thorner, Jesse},
  title   = {Exceptional zeros of {Rankin--Selberg} {$L$}-functions and joint {Sato--Tate} distributions},
  journal = {Int. Math. Res. Not. IMRN},
  year    = {2025},
  number  = {20},
  pages   = {rnaf307},
  doi     = {10.1093/imrn/rnaf307}
}

@article{Wong2019,
  author  = {Wong, Peng-Jie},
  title   = {On the {Chebotarev--Sato--Tate} phenomenon},
  journal = {J. Number Theory},
  volume  = {196},
  year    = {2019},
  pages   = {272--290},
  doi     = {10.1016/j.jnt.2018.09.010}
}

@misc{Woo2026,
  author        = {Woo, Katharine},
  title         = {Sums of {Hecke} eigenvalues along polynomial sequences and base change for {$\mathrm{GL}(2)$}},
  year          = {2026},
  note          = {arXiv:2604.18923},
  eprint        = {2604.18923},
  archivePrefix = {arXiv},
  primaryClass  = {math.NT}
}

@article{blomersums2008,
  author  = {Blomer, Valentin},
  title   = {Sums of {Hecke} eigenvalues over values of quadratic polynomials},
  journal = {Int. Math. Res. Not. IMRN},
  year    = {2008},
  number  = {16},
  pages   = {rnn059},
  note    = {29 pp.},
  doi     = {10.1093/imrn/rnn059}
}

@misc{kuan2023sums,
  author        = {Kuan, Chan Ieong and Lowry-Duda, David and Walker, Alexander and Steiner, Raphael S.},
  title         = {Sums of cusp form coefficients along quadratic sequences},
  year          = {2023},
  note          = {arXiv:2301.11901},
  eprint        = {2301.11901},
  archivePrefix = {arXiv},
  primaryClass  = {math.NT}
}

\end{document}